\documentclass[letterpaper,10pt]{article}
\usepackage{
amsmath,epsfig, amscd, makeidx, 
xspace, 
color, stmaryrd
}

\usepackage[english]{babel}                   
\usepackage[latin1]{inputenc} 
\usepackage{amsmath,amssymb}
\usepackage{epsfig}
\usepackage[all]{xy}
\xyoption{2cell}
\usepackage{theorem}

\newcommand{\eg}{e.g.,\xspace}
\newcommand{\ie}{i.e.,\xspace}

\newcommand{\smiley}{$\xy*!<0pt,-1.5pt>{..}
 *\cir<5pt>{}
 *!<0pt,0.7pt>\cir<3pt>{d^u}
 \endxy$}

\CompileMatrices

\newcommand{\bZ}{{\mathbb Z}}
\newcommand{\Q}{{\mathbb Q}}

\newcommand{\cC}{{\mathcal C}}
\newcommand{\cO}{{\mathcal O}}
\newcommand{\EE}{{\mathcal E}}

\newcommand{\ab}{{\ensuremath{\mathcal Ab}}}

\newcommand{\trivgroup}{0}
\newcommand{\p}{\widehat{{}_p}}

\newcommand{\holim}[1][{}]{\operatornamewithlimits{holim}_{\overleftarrow{#1}}}
\newcommand{\colim}[1][{}]{\operatornamewithlimits{lim}_{\overrightarrow{#1}}}
\newcommand{\hocolim}[1][{}]{\operatornamewithlimits{holim}_{\overrightarrow{#1}}}

\newcommand{\Ker}{\operatorname{Ker}}

\newcommand{\fib}{\twoheadrightarrow}

\newcommand{\smsh}{{\wedge}}

\newcommand{\sens}{\mathcal S_*}

\newcommand{\sen}{\mathcal S}

\newcommand{\gs}{\Gamma\!{\sens}}

\newcommand{\wefib}{\overset\sim\fib}

\newcommand{\ess}{\mathbf S}

\newcommand{\I}{\mathcal I}

\newcommand{\T}{\mathbb T} 

\theoremstyle{plain}
{
\theorembodyfont{\rmfamily}
\newtheorem{note}[subsubsection]{Note}
\newtheorem{Def}[subsubsection]{Definition}
\newtheorem{ex}[subsubsection]{Example}

}
\newtheorem{theo}[subsubsection]{Theorem}
\newtheorem{prop}[subsubsection]{Proposition}
\newtheorem{lemma}[subsubsection]{Lemma}
\newtheorem{cor}[subsubsection]{Corollary}
\newtheorem{remark}[subsubsection]{Remark}

\newenvironment{proof}{\par\noindent{\it Proof: }}{\qed\par}
\def\qedbox{$\Box$}
\newbox\QedBox \setbox\QedBox\vbox{\hrule height 1ex width .618ex}
\def\qedbox{\copy\QedBox}

\makeatletter
\def\qed{%
   {%
      \unskip
      \nobreak \hfil
      \penalty 50               
      \hskip 3em                
      \null \nobreak \hfil
      \qedbox
      \parfillskip=\z@skip
      \finalhyphendemerits=\z@  
      \endgraf                  
   }}
\makeatother

\newcommand{\Witt}{\mathbb{W}} 
\newcommand{\Group}{G}

\newcommand{\A}{\mathcal{A}}

\newcommand{\spaces}{\mathcal{S}_*}
\newcommand{\uspaces}{\mathcal{S}}
\newcommand{\M}{\mathcal{M}}
\DeclareMathOperator{\Aut}{Aut}

\newcommand{\Z}{\mathbb{Z}}  
\newcommand{\R}{\mathbb{R}}  

\newcommand{\Ar}{\mathcal{A}}

\newcommand{\Fin}{\mathrm{Fin}}

\newcommand{\id}{\mathrm{id}}
\newcommand{\Sl}{S^1}

\newcommand{\xto}{\xrightarrow}
\newcommand{\xgets}{\xleftarrow}
\newcommand{\CC}{\mathcal C}
\newcommand{\BB}{\mathcal B}
\newcommand{\Rel}{Rel}
\newcommand{\Cat}{Cat}

\newcommand{\diag}{p}
\newcommand{\xycoprod}[1][{}]{\underset{#1}{\coprod}}

\newcommand{\TZ}{\T^{\times n}}

\newcommand{\lont}{transformation\xspace}

\newcommand{\THH}{THH}
\newcommand{\HH}[1][{}]{HH^{#1}}

\newcommand{\TS}[1][{}]{\Lambda_{#1}}

\newcommand{\dT}[1][{}]{\Lambda_{#1}}
\newcommand{\ddT}{\Lambda}

\newcommand{\Kd}{\widetilde}

\usepackage[ps2pdf]{hyperref}
\begin{document}{
\author{
Morten Brun
\and
Gunnar Carlsson
 \thanks{The second author was supported in part
   by NSF DMS 0406992.  Part of this work was done while the second author visited
   University of Bergen, and we want to thank the institution for its hospitality}
 \and 
 Bj{\o}rn Ian Dundas
\thanks{Part of this work was done while the third author visited
   Stanford university, and we want to thank the institution for its
 hospitality. Part of this work was done while the second and third authors visited Institut
   Mittag-Leffler for the program in algebraic topology in the spring
   of 2006, and we want to thank the organizers for the
   invitation and the possibility to work there.}
} 
\title{Covering homology
}

\maketitle
\abstract{
We introduce the notion of {\em covering homology} of a commutative
$\ess$-algebra with respect to certain families of coverings of
topological spaces. The construction of covering homology is extracted
from B\"okstedt, Hsiang and Madsen's topological cyclic homology. In
fact covering homology with respect to the family of orientation
preserving isogenies of the circle is equal to topological cyclic
homology.
Our basic
   tool for the analysis of covering homology is a cofibration
   sequence involving homotopy orbits and a restriction map similar to
   the restriction map used in B\"okstedt, Hsiang and Madsen's
   construction of topological cyclic homology.

Covering homology with respect to families of isogenies of a torus is
 constructed from iterated topological
 Hochschild homology. It receives a trace map from iterated algebraic
 K-theory and the hope is that the rich structure, and the calculability
 of covering homology will make covering homology useful in the exploration
 of J. Rognes' ``red shift conjecture''.} 
\section{Introduction}
 Topological cyclic homology ($TC$), as  defined by B\"okstedt, Hsiang
 and Madsen in \cite{BHM}, is interesting for two reasons: firstly it is a
 good approximation to algebraic K-theory, secondly it is accessible through
 methods in stable homotopy theory.  Along with motivic
 homotopy theory, topological cyclic homology is the main source for
 calculations of algebraic K-theory.

 Topological cyclic homology is built from a diagram of categorical fixed point
 spectra of B\"okstedt's topological Hochschild homology.
 The main reason for the accessibility of $TC$ is the so-called
 ``fundamental cofiber sequence'' which inductively gives 
 homotopical control of the categorical fixed points.  There are many
 frameworks where people find conceptual reasons for the fundamental
 cofiber sequence -- for instance it can be viewed as a concrete
 identification of the geometrical fixed points -- but regardless of point of view it remains a
 marvellous fact at a crucial point of the theory.

 Just as for other cyclic nerve constructions, if the input is
 commutative -- in our case a connective commutative $\ess$-algebra $A$ --
 topological Hochschild homology extends to a functor of spaces $X\mapsto
 \dT[X]A$, where the value at the circle $S^1$ recovers the usual
 definition $\dT[{S^1}]A\simeq THH(A)$. 
If $X$ is a finite set,
 $\dT[X]A$ is just a particular model for the $X$-fold smash product
 of $A$ with itself. Our preferred model $\dT[X] A$ 
is extracted from B\"okstedt's construction of topological
 Hochschild homology 
and Street's first
 construction 
 \cite{MR0347936} to enhance the functoriality of homotopy colimits.
This functoriality has the side effect that the multiplicative
structure of topological Hochschild homology can be realized on our
concrete model, and using B\"okstedt's construction as our basis, we
get that $\dT[X]A$ is automatically a homotopy functor in both $X$ and
$A$ withouth any cofibrant replacements.

The main reason for our choice of model is that the
 fundamental cofiber sequence extends in a beautiful manner (see Lemma
 \ref{lem:fundcofseq1}) giving full homotopy theoretic control over
 the categorical fixed points. Currently
 it is not clear how to extend it working with other models for the smash
 product of commutative ringspectra. The sequence becomes particularly
 transparent in the abelian case, which is the interesting part if one
 is mostly concerned with the case $X$ being a torus (which is the
 case for iterated topological Hochschild homology): if $G$ is a finite abelian group,
 $X$ a non-empty free $G$-space and $A$ a connective commutative
 $\ess$-algebra, then (\ref{lem:funcofseq}) there is a cofiber sequence
  \begin{displaymath}
    [\dT[X](A)]_{hG}\to[\dT[X](A)]^G \to \holim[0 \ne H \le G] [\dT[{X/H}](A)]^{G/H},
  \end{displaymath}
where $[\dT[X](A)]_{hG}$ denotes the homotopy $G$-orbits,
$[\dT[X](A)]^G$ the categorical fixed points and the homotopy limit is
taken over all nontrivial subgroups of $G$.  In the equivariant world,
this could be viewed as an instance of the tom Dieck filtration gotten
by taking fixed points of the sequence one gets by smashing
$\dT[X](A)$ with the
cofibration sequence $EG_+\to S^0\to\widetilde{EG}$, as
discussed in \cite{carlssontopology}, together with identifications,
firstly of the
geometric $H$-fixed-points of $\dT[X](A)$ and  $\dT[{X/H}](A)$, and
secondly of the categorical fixed points  of $\dT[X](A)$ and the fixed points obtained by
deloopings by representations in a universe.  However, we also get
that these deloopings are not necessary for developing the
theory (with the exception of matters related to transfers, which
will be important in a later paper), and we can stay with the concrete functorial model at hand and
its associated categorical constructions.

 More precisely, by induction on the order of the group, the fundamental cofibration
 sequence imply that we have full homotopical control over the
 categorical fixed points $(\Lambda_XA)^G$ 
if $G$ is a finite group
 acting freely on $X$. 
 
As a result of this structure we get that if $X$ is connected, then
(\ref{prop:pi0offix}) there is a
natural isomorphism 
$$\pi_0 [\dT[X]A]^G\cong \Witt_G(\pi_0 A)$$ where the right hand side
is the Burnside-Witt ring of Dress and Siebeneicher \cite{DressSieb};
and we recover
Hesselholt and Madsen's result $\pi_0 [THH(A)]^{C_r} \cong
\Witt_{C_r}(\pi_0 A)$ \cite{HM1} where $C_r$ is the cyclic group of
order $r$.

 Studying systems of coverings, the spectra $(\dT[X]A)^G$ assemble into a
 diagram giving rise to the new notion of ``covering homology''.  In
 the particular case of finite orientation preserving self-coverings of
 the circle this is B\"okstedt, Hsiang
 and Madsen's topological cyclic homology.  
If we include reflections we get a definition of topological dihedral homology.

 In the special case where $X$ is the $n$-torus $\TZ$, the spectrum
 $\dT[\TZ]A$ is a model for the $n$-fold iterated topological
 Hochschild homology of $A$.  This said, the covering homology is very
 different from iterated topological cyclic homology, having a vastly richer
 structure.  We give examples at the very end of the paper where we
 see actions of various Galois groups, units in orders in division
 algebras (and so Morava
 stabilizer groups) and in the
 extreme case, all of $GL_n(\bZ)$.  The study of this structure and concrete calculations
 will be followed up in a second paper.  It is from this detailed
 analysis one should hope to glean insight into the chromatic
 behaviour of covering homology.

To give the reader an idea about the structure entering into the construction of covering homology,
consider topological Hochschild homology ($THH$) of the spherical group ring
${\ess}[G]$ of an abelian group $G$.  It turns out 
to be the suspension spectrum $\ess[Map(\T,B G)]$
of the free loop space of the classifying space of
the group in question, i.e. the space of unbased continuous maps of
the circle into $B G$.  The operators used to compute $TC$ arise from
the evident circle action on the free loop space, as well as from the
power maps of various degrees from the circle to itself.  This free
loop space interpretation of topological Hochschild homology  shows
that the $n$-fold iteration of $THH$ on $\ess[G]$ is equivalent to the
suspension spectrum 
$$\dT[\TZ](\ess[G])\simeq\ess[Map(\TZ,B^n G)]$$ on the 
unbased mapping space of a higher dimensional torus into the iterated
bar construction $B^nG$.  This
space supports many natural operations beyond the
one-variable ones.  For example, the group $GL_n ({\bZ})$ acts on the
$n$-torus, and hence on the mapping space from the torus into $B^nG$.
In addition, generalizations of the power maps include all
possible isogenies of the torus to itself.  This will be true in any
sufficiently functorial model for $THH$; the important point is that
the equivariant structure is ``right'' - a thing secured by the
fundamental cofiber sequence.


This is a brief overview of the paper. In sections 2--4 we construct
the Loday functor $X
\mapsto \Lambda_X A$. Section 2 is a guide to the construction
with references to related constructions. Section 3 contains 
combinatorial preliminaries, and in Section 4 we finally define 
$\Lambda_XA$.
In Section 5 we present the fundamental cofibration
sequence. In Section 6 this cofibration sequence is used to describe
the zeroth homotopy group of fixed points of $\Lambda_XA$ 
in
terms of the Burnside--Witt construction. Finally in the short Section
7 we define covering homology and show how it extends the definition
of topological cyclic homology.
\setcounter{tocdepth}{2}

\setcounter{subsection}{0}\setcounter{subsubsection}{0}
\section{Guide to the construction}

Let $X$ be a space.
In the sections to follow we give a model, $\dT[X]A$, for the $X$-fold smash-power
of a connective commutative $\ess$-algebra $A$ (\ie a symmetric monoid in
$(\gs,\smsh,\ess)$). We have chosen to
work with $\Gamma$-spaces, but our constructions 
work equally well on connective commutative symmetric ring-spectra.
We shall call $\dT[X]A$ the {\em Loday functor} of
$A$ evaluated at $X$. It is
important for the
construction that $A$ is {\it strictly} commutative.
The model $\dT[X]A$ is functorial in both 
$X$ and $A$. Therefore, if a group $G$ acts on $X$, then we can
consider the (categorically honest) $G$-fixed points of $\dT[X]A$. In
the particular situation 
where $G$ is finite and the action on $X$ is free 
we have good control on the fixed point spectrum $(\dT[X]A)^G$.
When $X$ is the underlying space of the simplicial circle group $\T =
\sin U(1)$, given by the singular complex on the
circle group $U(1) = \{x \in \mathbb C \colon |x|=1\}$, the Loday functor
evaluated at $X$ is a model 
for topological Hochschild homology.

Notation: $\Delta$ is the category of finite non-empty ordered sets,
$\Delta^o$ its opposite, $Fin$ is the category of finite sets,
$\Gamma^o$ is the category of finite pointed spaces and $\ab$ is the
category of abelian groups.

\subsection{Higher Hochschild homology}
\label{sec:itreahoch}
\label{noneqivTHH}
As a motivation, consider ordinary Hochschild homology
$HH(A)\colon\Delta^o\to\ab$ of a flat ring $A$, given in each
dimension by 
$$HH_q(A)=A^{\otimes q+1}.$$
If $A$ is commutative, then $\HH(A)$ is a simplicial commutative ring,
and Loday \cite{MR981743} 
observed that Hochschild homology factors
through the category $\Fin$ of finite sets: 
$$
\xymatrix{\Delta^o\ar[rr]^{HH(A)}\ar[d]^{\Sl}&&{\ab}\\{\text{Fin}}\ar@{.>}[urr]}.
$$
(see \eg \cite{Loday}). 
Here $\Sl = \Delta[1]/\partial \Delta[1]$ is the standard simplicial circle.
The diagonal functor is oftentimes called the {\em Loday complex}.  Let us write
simply 
$$X\mapsto \dT[X]^{\Z} A$$ for the diagonal functor. 
This is  functorial in
the finite set $X$, and so, if we extend to all sets by colimits and
to all simplicial sets by applying the functor degreewise, we get a
functor $X\mapsto \dT[X]^{\Z} A$ from the category $\sen$ of spaces
(simplicial sets) to simplicial abelian groups, with classical Hochschild
homology being  $$\HH(A) = \dT[{\Sl}]^{\Z} A.$$
Pirashvili
\cite{MR1755114} uses
the notation $H^{[d]}(A,A)$ for the Locay complex for $A$ evaluated on
the $d$-dimensional sphere 
and calls it
  ``higher Hochschild homology of order $d$''. 

There is no obstruction to apply the same construction to symmetric monoids in any symmetric monoidal category, Hochschild homology being the case when one considers $(\ab, \otimes,\bZ)$.

\subsection{Higher topological Hochschild homology}
\label{sec:iterthh}

Consider any of the popular symmetric monoidal categories of spectra.
Then topological Hochschild homology of a cofibrant $\ess$-algebra $A$
is equivalent to the simplicial spectrum gotten by just replacing
$\otimes$ with $\smsh$ in the definition of Hochschild homology, and
in the commutative case we have a factorization $X \mapsto A \otimes
X$ through $Fin$, where  $\otimes$ is
  the categorical tensor in commutative $\ess$-algebras. Extending
  this functor to the category of spaces just as in the above
  subsection we obtain for every commuative $\ess$-algebra $A$ and
  every space $X$ the {\em higher topological Hochschild homology} $A
  \otimes X$. The above 
is another way of stating the result of McClure, Schw{\"a}nzl and Vogt 
  \cite{McClureTHH}: $THH(A)\simeq A\otimes \Sl$. Here $THH$ is
  B\"okstedt's model for topological Hochschild homology. The above
  equivalence is
  an equivalence of cyclic spectra. However 
  in the context of
  $\ess$-algebras in the sense of Elmendorf, Mandell, Kriz and May
  \cite{EKMM} the fixed point spectrum $(A\otimes sd_r \Sl)^{G}$ of
  $A\otimes sd_r \Sl$
  with respect to the cyclic group $G$ with $r > 1$ elements does not have
  the same homotopy type as the $G$-fixed point 
  spectrum of $sd_r \THH(A)$. 
In the language of McClure et al. it is easy to see what the iterated topological Hochschild homology is:
$$THH(THH(A))\simeq (A\otimes \Sl)\otimes \Sl\simeq A\otimes(\Sl\times \Sl).$$
Hence taking the $n$th iterate of $THH$ is the same as tensoring with the $n$-torus.



 

  
In the situation where $X = \T$ is the circle group and $G$ is a
finite subgroup of $\T$ there is a homotopy
equivalence between the Loday functor $\ddT_{\T}A$ of $A$ evaluated at
$\T$ and $THH(A)$, and this equivalence is $G$-equivariant in the
sense that it induces an equivalence of $H$-fixed spectra for every
subgroup $H$ of $G$. In
this situation the 
control on the fixed point space derives from the
``fundamental cofiber sequence'' 
$$THH(A)_{hC_{p^n}}\to THH(A)^{C_{p^n}}\to THH(A)^{C_{p^{n-1}}}.$$

In Lemma \ref{lem:funcofseq} we generalize the fundamental cofibration
sequence  
to the situation where a finite group $G$ acts freely
on a space $X$.  
In the toroidal case $X = \TZ$, the spectrum $\ddT_X A$
is homotopy equivalent to iterated topological Hochschild homology, but
the fact that the 
complexity of the
subgroup lattice of torus increases with dimension
makes the fundamental cofibration sequence more involved.  
This also gives rich and interesting symmetries
on the collection of fixed point spectra $\ddT_{\TZ}(A)^H$ under varying finite
$H\subseteq \TZ$.  This structure gives actions by interesting groups
possibly shedding light 
on the chromatic properties of the spectra, as in Rognes' red shift
conjecture. 

Although iterated $TC$
involves iterations of fixed point spectra of $THH$, we consider
the much simpler idea of taking the fixed point under the toroidal
action on the iterated $THH$.  This gives us in many ways a much
cruder invariant, but also a much more computable one -- essentially,
the difference is that of $\left((A\smsh A)^{C_2}\smsh (A\smsh
A)^{C_2}\right)^{C_2}$ (the start of the iterated construction) and $\left((A\smsh A)\smsh(A\smsh
A)\right)^{C_2\times C_2}$ (the start of our construction).  Taken $A
= \ess$ we see that these constructions give different results. Furthermore,
the resulting theory displays a vastly richer symmetry, giving rise to
a plethora of actions not visible if one only focuses on ``diagonal''
actions.

The underlying
spectra of  $X \mapsto \ddT_X A$ and $A \otimes X$ are equivalent.
The problem with the model $A \otimes X$ is that if a finite group $G$ acts on
$X$ then we do not fully understand the $G$-fixed points
of this model.
Presumably this could be fixed along the lines of Brun and Lydakis'
version of $THH$ (unpublished) or Tore Kro's thesis \cite{kro}.
However, for now we choose a
hands-on approach.



\setcounter{subsection}{0}\setcounter{subsubsection}{0}

\setcounter{subsection}{0}\setcounter{subsubsection}{0}

\section{Encoding coherence with spans}
\label{sec:encodespans}

In this section we shall use the category $V$ of spans of finite sets
described below,
to encode the coherence data for symmetric monoidal categories. More
precisely, we shall encode the coherence data by a lax functor
from $V$ to the category of categories. Since we are interested in
group actions we will also investigate group actions on the morphism
sets of $V$. The reader who is willing to believe in our coherence results
may skip this quite technical section. However these coherence
results are essential for the construction of the Loday functor
$\dT$ in the next section.
 
\subsection{The category of spans}
\label{sec:eqspans}

The
object class of the category $V$ of spans of finite set is the class
of finite sets.
Given finite sets 
$X$ and $Y$ the set of morphisms $V(Y,X)$ is the set of equivalence classes $[Y
\gets A \to X]$ of
diagrams of finite sets of the form $Y \gets A \to X$. Here $Y \gets A \to X$ is
equivalent to $Y \gets A' \to X$ if there exists a bijection $A \to
A'$ making the resulting triangles commute. The composition of two
morphisms $[Z \gets B \to Y]$ and $[Y \gets A \to X]$ is the morphism
$[Z \gets C \to X]$, where $C$ is the pull-back of the diagram $B \to
Y \gets A$. 

If $\Group$ is a group acting on a finite set $X$, then $\Group$
acts on $X$ considered as an object of $V$ through the functor $g
\mapsto g_*$. Consequently, $\Group$ acts on the morphism sets
$V(Y,X)$. There is a function 
\begin{displaymath}
  \varphi \colon V(Y,X/\Group) \to V(Y,X)^{\Group}
\end{displaymath}
defined by the formula
\begin{displaymath}
 \varphi([Y \gets A \to X/\Group]) = [Y \gets A \times_{X/\Group} X  \to X].
\end{displaymath}

\begin{prop}
  The function $\varphi$ is bijective.
\end{prop}
\begin{proof}
  The heart of the proof of this result is that the category $V$
has a categorical product given by the disjoint union of finite
sets. This gives rise to a $\Group$-equivariant bijection
\begin{displaymath}
  V(Y,X) \cong \prod_{x \in X} V(Y,\{x\}) \cong \Fin(X,V(Y,*)), 
\end{displaymath}
where $\Group$ acts trivially on $V(Y,*)$ and on $\Fin(X,V(Y,*))$ by
conjugation. It is not hard to check that the composition
\begin{displaymath}
  V(Y,X/\Group) \cong \Fin(X/\Group,V(Y,*)) \cong
  \Fin(X,V(Y,*))^{\Group} \cong V(Y,X)^{\Group}
\end{displaymath}
is the map $\varphi$. 
\end{proof}

\begin{cor}
\label{fixfix}
  There is an isomorphism of categories $\psi \colon V/(X/\Group) \to (V/X)^{\Group}$.
\end{cor}
\begin{proof}
  The functor $\psi$ takes an object $\alpha \in V(Y,X/\Group)$ to
  $\varphi(\alpha) \in V(Y,X)^G$. By the definition of $\varphi$ it is
  clear that it is functorial in the sense that the diagram
  \begin{displaymath}
    \begin{CD}
      V(Z,Y) \times V(Y,X/\Group) @>>> V(Z,X/\Group) \\
      @VV{\id \times \varphi}V @VV{\varphi}V \\
      V(Z,Y) \times V(Y,X)^{\Group} @>>> V(Z,X)^{\Group}
    \end{CD}
  \end{displaymath}
  commutes for all finite sets $Y$ and $Z$. Therefore, 
  on morphisms, we can define the functor $\psi$ to be
  given by the identity.
\end{proof}

To get the multiplicative structure on our construction we will need
the following lemma
\begin{lemma}\label{Vismonoidal}
  Disjoint union of sets is the coproduct on $V$.
\end{lemma}
\begin{proof}
  To be precise, the disjoint union of two maps $[X\gets A\to Y]$ and
  $[X'\gets A'\to B']$ is $[X\sqcup X'\gets A\sqcup A'\to Y\sqcup
  Y']$.  The required isomorphism $V(X,Y)\times V(X',Y)\cong
  V(X\sqcup X',Y)$ is given by 
  \begin{multline*}
      \left([\xymatrix{X&A\ar[l]_f\ar[r]^g&Y}],[\xymatrix{X'&A'\ar[l]_{f'}\ar[r]^{g'}&Y}]\right)\qquad
\\ \mapsto
    \qquad [\xymatrix{X\sqcup X'&A\sqcup A'\ar[l]_-{f\sqcup
  f'}\ar[r]^-{g+g'}&Y}] 
  \end{multline*}
with inverse given by the restrictions to the appropriate inverse images
  \begin{multline*}
      [\xymatrix{X\sqcup X'&A\ar[l]_-{f}\ar[r]^g&Y}]\qquad
\\ \mapsto\qquad \left([\xymatrix{X&f^{-1}(X)\ar[l]\ar[r]&Y}],[\xymatrix{X'&f^{-1}(X')\ar[l]\ar[r]&Y}]\right). 
  \end{multline*}
\end{proof}

\subsection{A lax functor}
\label{sec:laxfun}

In this section we shall work heavily with the concepts around
bicategories. Unfortunately the relevant terminology is not universal,
and since the book \cite{MR2094071} of
Leinster is freely available on the arXiv we shall use the
notation he uses throughout. 

Given a bicategory $\BB$ and $0$-cells $A$ and $B$ there is a category
$\BB(A,B)$. The objects of $\BB(A,B)$ are the $1$-cells from $A$ to
$B$, and the morphisms in $\BB(A,B)$ are $2$-cells in $\BB$.

  Let $W$ denote the following bicategory of spans (compare \cite[pp. 283--285]{McL}). The bicategory $W$ has the class of finite sets as class
  of zero-cells, and $W(X,Y)$ is the 
  category of spans $f = (X \xgets {f_1} A \xto {f_2} Y)$. A morphism
  $\alpha \colon f \to g$ from $f$ to $g = (X
  \xgets {g_1} A' \xto {g_2} Y)$ in $W(X,Y)$, that is, a $2$-cell
  $\alpha$ in $W$, consists of a
  map from $\alpha \colon A \to A'$ making the diagram
\begin{equation}\label{twocell}
  \xymatrix@!0{
&& A \ar[drr]^{f_2} \ar[dll]_{f_1} \ar[dd]_{\alpha} &&\\
X & & && Y\\
& &A' \ar[urr]^{g_2} \ar[ull]_{g_1} &&
}    
\end{equation}
commute. 
The composition functor
  \begin{displaymath}
    W(Y,Z) \times W(X,Y) \to W(X,Z)
  \end{displaymath}
  takes a pair $(g,f)$
  of spans $Y \xgets {g_1} B \xto {g_2} Z$ and $X 
  \xgets {f_1} A \xto {f_2} Y$ to the span $g \circ f$ represented by
  the diagram 
  \begin{displaymath}
    X \xgets {(g \circ f)_1} A \times_Y B \xto {(g \circ f)_2} Z
  \end{displaymath}
  where 
  $$A \times_Y B = \{(a,b) \in A \times B \, \colon \, f_2(a) =
  g_1(b) \}$$ 
  is a functorial choice of fiber product.
  Here $(g \circ f)_1 (a,b) = f_1(a)$ and $(g \circ f)_2(a,b) =
  g_2(b)$.
  On morphisms the composition functor is defined by the same
  formula. The identity functor $\mathbf 1 \to W(X,X)$ is given by the 
  one-cell $X \xgets = X \xto = X$. Using the universal property of
  pull-backs it is easy to verify 
  that $W$ is a bicategory.

The (strict) bicategory $\Cat$ has small categories as $0$-cells, functors as
$1$-cells and natural transformations as $2$-cells. 

In the notation we
have adopted a {\em weak functor} $F \colon \BB \to \CC$ is a lax
functor with the property that the $2$-cells $\phi_{g,f} \colon Fg \circ Ff \to F(g
\circ f)$ and $\phi_A \colon \id_{FA} \to F(\id_A)$ are
isomorphisms. At many places weak functors are called pseudo-functors.
A {\em strict functor} $F \colon \BB \to \CC$ is 
a lax
functor with the property that the $2$-cells $\phi_{g,f} \colon Fg \circ Ff \to F(g
\circ f)$ and $\phi_A \colon \id_{FA} \to F(\id_A)$ are identity
morphisms.

\begin{Def}
\label{defineendoL}
  We define a lax functor $L \colon W \to W$.
  On
  $0$-cells $L$ is the identity. Let $f = (X \xgets {f_1} A \xto {f_2}
  Y)$ be a $1$-cell in $W$. We define
  \begin{displaymath}
    L(f) = X \xgets {Lf_1} LA \xto
    {Lf_2} Y, 
  \end{displaymath}
  where $LA$ is the image of the map $(f_1,f_2) \colon A \to X \times
  Y$, that is,
  \begin{displaymath}
    LA = \{(x,y) \in X \times Y \, : \, \exists a \in A \text{ such
      that } f_1(a) = x, \text{ and } f_2(a) = y \}.
  \end{displaymath}
  The maps $Lf_1$ and $Lf_2$ are the projection maps to the first and
  the second factor. Given a $2$-cell $\alpha$ of the form
  (\ref{twocell}) in $W$, the
  set $LA$ is a subset of $LA'$, and we define $L\alpha$ to be the
  $2$-cell
\begin{equation*}
  \xymatrix@!0{
&& LA \ar[drr]^{Lf_2} \ar[dll]_{Lf_1} 
\ar[dd]^-{\text{incl.}} &&\\
X & & && Y\\
& &LA' \ar[urr]_{Lf'_2} \ar[ull]^{Lf'_1} &&
}    
\end{equation*}
given by the inclusion $LA\subseteq LA'$.
 Let $g$ be a $1$-cell in $W$ of the form $g = (Y
  \xgets {g_1} B \xto {g_2} Z)$. 
There is a natural invertible $2$-cell
\begin{displaymath}
    L(g) \circ L(f) \cong \left(X\gets\{ (x,y,z) \in X \times Y \times Z \, : \,
    (x,y) \in LA \text{ and } (y,z) \in LB\}\to Z\right),
  \end{displaymath}
where the maps are the projections, and precomposing with (the inverse
of) this $2$-cell,
the structure map
  \begin{displaymath}
    L(g,f) \colon L(g) \circ L(f) \to L(g \circ f)
  \end{displaymath}
  corresponds to the projection taking $(x,y,z)$ to $(x,z)$.
\end{Def}

If $f_1$ is the identity, then
\begin{displaymath}
  LA = \{ (x,f_2(x)) \in X \times Y \, \colon x \in X\}
\end{displaymath}
is the graph of $f_2$ and $A \cong LA$. Note that if also $g_1$ is the
identity, then the structure map $L(g,f) \colon L(g) \circ L(f) \to
L(g \circ f)$ is an identity morphism. Thus considering $\Fin$ as a
bicategory with only identity $2$-cells, there is a strict functor
$\Fin \to W$ which is the identity on $0$-cells and which takes a map $f
\colon X \to Y$ to the $1$-cell $(X \xgets{=} X \xto f Y)$.

The following lemma is straight forward
\begin{lemma}\label{Lpreservescoproducts}
  Considered as an endofunctor in $V$, $L$ preserves coproducts.
\end{lemma}

We can use the lax functor $L$ to define a bicategory $\Rel$ of
relations. The $0$-cells of $\Rel$ is the class of finite sets.
Given finite sets $X$ and $Y$ the category $\Rel(X,Y)$
is the full subcategory of 
$W(X,Y)$ consisting of $1$-cells in $W$ of the form $f = (X \xgets
{f_1} A \xto {f_2} Y)$, where $A$ is a subset of $X \times Y$. The
composition $g \circ_{\Rel} f$ of two $1$-cells in $\Rel$ is the $1$-cell $L(g \circ f)$,
where $g \circ f$ is the compositon of $f$ and $g$ considered as
$1$-cells in $W$. It is easily verified that this defines a bicategory $\Rel$.

\begin{remark}
  The lax functor $L \colon W \to \Rel$ is left adjoint to the inclusion
  $\Rel \subseteq W$ in the following sense:
  A lax functor $F \colon \BB \to \CC$ is {\em left adjoint} to the lax functor $G \colon \CC
  \to \BB$ if there exist a transformation $\beta \colon 1 \to G \circ F$
  and a transformation $\gamma \colon F \circ G \to 1$ such
  that the composites
  \begin{displaymath}
    \BB(A,GX) \xto{F} \CC(FA,FGX) \xto{\CC(\id_{FA},\gamma_{X})} \CC(FA,X)
  \end{displaymath}
  and
  \begin{displaymath}
    \CC(FA,X) \xto{G} \BB(GFA,GX) \xto{\BB(\beta_A,\id_{GX})} \BB(A,GX)
  \end{displaymath}
  are adjoint functors for all $0$-cells $A$ in $\BB$ and $X$ in $\CC$.  
\end{remark}

Note that $L \colon W \to \Rel$ is a strict functor and that if
$\alpha$ is an invertible $2$-cell in $W$, then $L\alpha$ is an
identity $2$-cell in $\Rel$. This implies that there is a strict
functor $L_0 \colon V \to \Rel$ taking a morphims, that is, $1$-cell $f$ to
$Lf$. On the other hand the
inclusion $\Rel \subseteq W$ is only a lax functor.
\begin{Def}
\label{defineT}
  Consider the category $V$ as a bicategory with only identity
  $2$-cells. The lax functor $T \colon V \to W$ is defined to be the composition
  \begin{displaymath}
    V \xto{L_0} \Rel \xto{\subseteq} W
  \end{displaymath}
  of the strict functor $L_0$ and the inclusion of $\Rel$ into $W$.
\end{Def}

The functors $Fin^o \to V$, $(f
\colon X \to Y) \mapsto 
f^* = [Y \xleftarrow{f} X \xto{=} X]$  and $Fin \to V$, $(g \colon Y
\to X) \mapsto g_* = [Y \xleftarrow{=} Y \xto{g} X]$ allow us to
consider both $\Fin$ and $\Fin^o$ as subcategories of $V$.
Note that the composite lax functors $f \mapsto T(f^*)$ and $f \mapsto
T(f_*)$ are both strict functors. By abuse of notation we shall write
$T \colon \Fin^o \to W$ and $T \colon \Fin \to W$ for these strict
functors. 

\begin{Def}
\label{definepseudo}
  Let $\CC$ be a category with finite coproducts. We define a weak functor $\CC \colon W \to Cat$.
  Given a finite set $X$ we let $\CC(X) = \CC^X$ be the
  $X$-fold product of $\CC$. Given a one-cell $f = (X \xgets {f_1} A
  \xto {f_2} Y)$ the functor $\CC(f) \colon \CC(X) \to \CC(Y)$ is
  given by the formula
  \begin{displaymath}
    (\CC(f)(c))(y) = \coprod_{a \in f_2^{-1}(y)} c(f_1(a)).
  \end{displaymath}
  Given a two-cell of the form \ref{twocell},
we define $\CC(\alpha) \colon \CC(f) \to \CC(f')$ to be the natural transformation 
with $(\CC(\alpha)(c))(y)$ equal to the composition
\begin{displaymath}
  \coprod_{a \in f_2^{-1}(y)} c(f_1(a)) \cong \coprod_{a' \in
    (f_2')^{-1}(y)} \coprod_{a \in \alpha^{-1}(a')} c(f'_1(a')) \to 
  \coprod_{a' \in (f'_2)^{-1}(y)} c(f'_1(a')),
\end{displaymath}
where the last map is the sum map.
\end{Def}

Composing the weak functor $\CC \colon W \to \Cat$ with the lax
functor $T \colon V \to W$ we obtain a lax functor $\CC T \colon V \to
\Cat$, which we by abuse of notation refer to simply as $\CC \colon V
\to \Cat$. Composing with the inclusion $\Fin \subseteq V$ 
we obtain a weak functor
\begin{eqnarray*}
 \CC \colon \Fin \to \Cat, \qquad
 X \xto f Y \quad \mapsto \quad \CC^X &\xto {\CC T(f)}&
 \CC^Y,\\ 
c &\mapsto& y \mapsto \coprod_{x \in f^{-1}(y)} c(x).
\end{eqnarray*}
On the other hand, composing with the inclusion $\Fin^o \subseteq V$ we obtain a strict functor
\begin{eqnarray*}
 \CC \colon \Fin^o \to \Cat, \qquad
 X \xgets f Y \quad \mapsto \quad \CC^X &\xto {\CC T(f)}&
 \CC^Y,\\ 
c &\mapsto& c \circ f.
\end{eqnarray*}

\section{The Loday functor}
\label{sec:zx}
In this section we shall construct the Loday functor $\ddT_X
A$. It is a functor in the unpointed space $X$ and the commutative
$\ess$-algebra $A$. When $X$ is the circle group, this a version
of topological Hochschild homology, which we think of as the cyclic
nerve of $A$ in the category of $\ess$-modules.

The construction of $\dT[X]A$ is quite technical. It can be
summarized in the following steps.
The crucial ingredient of B{\"o}kstedt's definition of topological
Hochschild homology is the stabilizations indexed by the skeleton category $\I$
of the category of finite sets and injective functions.
Firstly,
we note that $\I$ is a category with finite coproducts, and we apply the
results of the previous section to obtain 
an op-lax functor $\I T \colon V \to \Cat$
from the
category of spans $X \gets A \to Y$ of finite sets to the category of categories.
Secondly, we
construct the transformation $G^A$ associated to a commutative $\ess$-algebra
$A$. This is a transformation from the lax
functor given by the composition $Fin \xto{T} W \xto \I \Cat$ to the
constant functor from $Fin$ to $\Cat$ with value $\gs$. Thirdly,
using Street's first construction we rectify $\I$ to a strict functor
$\widetilde \I \colon V \to \Cat$. Finally, applying the homotopy colimit
construction, we obtain a functor $X \mapsto \ddT_XA$ from $Fin$ to $\gs$.

\subsection{On the functoriality of homotopy colimits}
\label{funhocolim}
B\"okstedt's construction of topological Hochschild homology involves
in each degree a stabilization by means of a 
homotopy colimit over a certain category. When the degrees vary, so
do the categories, and since book-keeping 
becomes rather involved we offer an overview of the
functoriality properties of homotopy colimits. 

First some notation.
Given bicategories $\BB$ and $\CC$, the bicategory $Lax(\BB,\CC)$ has the
class of lax functors $F \colon \BB \to \CC$ as $0$-cells. The
$1$-cells of $Lax(\BB,\CC)$ are transformations and the $2$-cells are
modifications.  
The bicategory $[\BB,\CC] \subseteq Lax(\BB,\CC)$ has the
class of strict functors $F \colon \BB \to \CC$ as $0$-cells. The
$1$-cells of $[\BB,\CC]$ are strict transformations and the $2$-cells are
modifications.

Let $Cat$ be the category of small categories, 
and consider $J\in Cat$ as a bicategory with only identity
$2$-cells. Given strict functors $E,F \colon J \to Cat$ a
strict transformation $G \colon F \rightarrow E$ consists of the 
following data: for every $i\in J$ we have a
functor $G_i\colon F(i)\to E(i)$ and for every $f\colon i\to j\in J$ we
have a natural transformation $G_f\colon E(f) G_i\Rightarrow G_j F(f)$ such
that if $g\colon j\to k\in J$ we have that $G_{gf}=G_gG_f$ and $G_{id_j}=id_{G_j}$.
The category $[J,Cat]/E$ has as objects pairs $(F,G)$ of a functor $F \colon J
\to Cat$ and a transforamtion $G \colon F \rightarrow E$.
A
morphism $(F,G)\to (F',G')$ in $[J,Cat]/E$ consists of a natural
transformation $\epsilon\colon F\Rightarrow F'$ and a modification $\eta\colon
G\Rightarrow G'\epsilon$.  That $\eta$ is a modification means that
for each $i\in J$ we have a natural transformation $\eta_i\colon
G_i\Rightarrow G'_i\epsilon_i$ such that for each $f\colon i\to j\in
J$ we have $\eta_jG_f=G'_f\eta_i\colon G_i\Rightarrow
G'_j\epsilon_jF(f)=G'_jF'(f)\epsilon_i$.

Let $K$ be a category with all small
coproducts. Considering $K$ as a constant  functor from $J$ to (big)
categories, we can extend the above construction to give a category $[J,Cat]/K$.
The homotopy colimit is the functor
$$\hocolim\colon [J,Cat]/K\to [J\times\Delta^o,K]$$
sending $(F,G)$ to $\hocolim[F]G$, which is the functor taking
$(j,[q])\in J\times\Delta^o$ to
$$\coprod_{x_0\gets\dots\gets x_q\in F(j)}G_j(x_q).$$

If $(\epsilon,\eta)\colon (F,G)\to (F',G')$ is a morphism, we
write $\hocolim[\epsilon]\eta$ for the corresponding natural
transformation $\hocolim[F]G\Rightarrow \hocolim[F']G'$ given as
$$
\xymatrix{
  {\xycoprod[x_0\gets\dots\gets x_q\in F(i)]}G_j(x_q)\ar[r]^{\coprod\eta_i}&
  {\xycoprod[x_0\gets\dots\gets x_q\in F(i)]}G'_j\epsilon_i(x_q)\ar[r]&
  {\xycoprod[y_0\gets\dots\gets y_q\in F'(i)]}G'_j(y_q)
}
$$
where the last map is the obvious one.

\subsection{The transformation {$G^A$}}
\label{thelont}





Let $\Sigma$ be the 
category of finite sets and bijections
and choose a strong symmetric monoidal functor
$S \colon \Sigma \to \spaces$ with $S(i) = S^i$.
We may consider $S$ as a transformation
$$
\xymatrix{
Fin
\ar@{=}[r]\ar[d]^{\Sigma}\ar@{}[dr]|{\overset{S}{\Rightarrow}}&Fin \ar[d]^{\spaces}\\
Cat\ar@{=}[r]&Cat.}
$$
Likewise, given a commutative
$\ess$-algebra $A$, we can use the multiplication in $A$ to construct a \lont
$$
\xymatrix{
Fin
\ar@{=}[r]\ar[d]^{\Sigma}\ar@{}[dr]|{\overset{A}{\Rightarrow}}&Fin \ar[d]^{\spaces}\\
Cat\ar@{=}[r]&Cat.}
$$
with $A(j)\cong \bigwedge_{x\in
  X}A(S^{j(x)})$ for $j\in\Sigma^X$.
We shall use the transformations $S$ and $A$  to
construct a transformation {$G^A$} from 
the weak functor $\I \colon \Fin \to \Cat$ with $\I(S) = \I^S$ to the constant functor $Fin\to
Cat$ sending everything to $\gs$.


If $X$ and $Y$ are pointed simplicial sets, $Map_*(X,Y)$ denotes the
$\Gamma$-space given by sending $Z\in\Gamma^o$ to the space of pointed
maps
from $X$ to the singular complex of the geometric realization of $Z\smsh Y$.

For a finite set $T$, let 
$$G^A_T\colon \I^T\to\gs$$
be the functor which to the object $i\in\I^T$
assigns the $\Gamma$-space 
$$Map_*(S(i),A(i))\cong Map_*(\bigwedge_{t\in
  T}S^{i(t)},\bigwedge_{t\in T}A(S^{i(t)})),$$
 and to a morphism
$\alpha\colon i\to j$
assigns the map
$$
\begin{CD}
  Map_*(\bigwedge_{t\in T}S^{i(t)},\bigwedge_{t\in T}A(S^{i(t)}))\\@VVV\\
  Map_*(\bigwedge_{t\in T}S^{j(t)-\alpha i(t)}\smsh
  \bigwedge_{t\in T}S^{i(t)}, 
  \bigwedge_{t\in T}S^{j(t)-\alpha i(t)}\smsh\bigwedge_{t\in T}
  A(S^{i(t)}))\\@VVV\\
  Map_*(\bigwedge_{t\in T}S^{j(t)},\bigwedge_{t\in T}A(S^{j(t)})),
\end{CD}
$$
where the first map is suspension and the second map is given by the
isomorphism $$\bigwedge_{t\in T}S^{j(t)-\alpha i(t)}\smsh
\bigwedge_{t\in T}S^{i(t)}\cong \bigwedge_{t\in T}S^{j(t)}$$ and the
map
\begin{align*}
  \bigwedge_{t\in T}S^{j(t)-\alpha i(t)}\smsh\bigwedge_{t\in T}
  A(S^{i(t)})\cong &\bigwedge_{t\in T}\left(S^{j(t)-\alpha i(t)}\smsh
    A(S^{i(t)})\right)\\
  \to &\bigwedge_{t\in T}A(S^{j(t)-\alpha
    i(t)}\smsh S^{i(t)})\cong \bigwedge_{t\in T}A(S^{j(t)}).
\end{align*}

If $\phi\colon S\to T$ is a function of finite sets, there is a
natural transformation 
$$G^A_\phi\colon G^A_S\to G^A_T\circ \I(\phi)$$
\newdir{|>}{!/2pt/dir{|}
*:(1,-.2)\dir^{>}
*:(1,+.2)\dir_{>}}
$$
\xymatrix{
{\I^S}
\ar[rr]^{G^A_S}
\ar[d]_{{\I(\phi)}}
&\ar@{=|>}(7,-7)_>>{G^A_\phi}
&\gs\\
{\I^T}\ar[urr]_{G^A_T}}
$$
given on $i\in\I^S$ by sending a map
$f\colon S(i) \to A(i)$
to the composite
$$
\begin{CD}
  \bigwedge_{t\in T}S^{\coprod_{s\in\phi^{-1}(t)}i(s)}@.@.
  \bigwedge_{t\in T}A(S^{\coprod_{s\in\phi^{-1}(t)}i(s)})\\
  @A{\cong}AA @. @A{\cong}AA\\
\bigwedge_{t\in T}\bigwedge_{s\in\phi^{-1}(t)}S^{i(s)}@.
\bigwedge_{t\in T}\bigwedge_{s\in\phi^{-1}(t)}A(S^{i(s)})@>>>
\bigwedge_{t\in T}A(\bigwedge_{s\in\phi^{-1}(t)}S^{a_{i(s)}})\\
@A{\cong}AA @A{\cong}AA @.\\
\bigwedge_{s\in S}S^{i(s)}@>f>>\bigwedge_{s\in S}A(S^{i(s)}) @.
\end{CD},
$$
where the top vertical maps are structure maps for $S$, being a
transformation, 
and the unmarked arrow is multiplication in $A$.

This structure assembles into the fact that $G^A$ is a transformation 
from the weak functor $S\mapsto\I^S$ to the constant functor $Fin\to
Cat$ sending everything to $\gs$ (one has to check that for every
$\psi\colon R\to S\in Fin$ the natural
transformations
$$
\xymatrix{
  &{\I^R}\ar[rr]^{G^A_R}\ar[dd]^{{\I^{\phi\psi}}}\ar[dl]_{{\I^{\psi}}}
  &\ar@{=|>}(23,-14)_{G^A_{\phi\psi}}&\gs\\
  {\I^S}\ar[dr]_{{\I^{\phi}}}&\ar@{=|>}(4,-13)_{{\I(\psi,\phi)}}&\\
  &{\I^T}\ar[uurr]_{G^A_T}
}
\qquad\text{and}\qquad
\xymatrix{
  &{\I^R}\ar[rr]^{G^A_R}\ar[dl]_{{\I^{\psi}}}
  &\ar@{=|>}(14,-7)_>>{G^A_{\psi}}&\gs\\
  {\I^S}\ar[dr]_{{\I^{\phi}}}\ar[urrr]_{G^A_S}
  &\ar@{=>}[d]_{G^A_{\phi}}&\\
  &{\I^T}\ar[uurr]_{G^A_T}&
}$$
are equal - the unitality degenerates as $G^A_{id_S}=id_{G^A_S}$). 

\begin{note}
  The reader may wonder why $G^A$ is only defined on $Fin$ (rather
  than $V$).  The reason is that there exists no extension of $G^A$ to
  $V$, for if one did, a factorization $p\circ i\colon  \{1\}\subseteq\{1,2\}\fib\{1\}\in Fin$ of the identity would then induce a
  splitting 
$$G_{i^*}^\ess(p^*j)\circ G_{p^*}^\ess(j) \colon Map_*(S^2,S^2) \to
  Map_*(S^4,S^4)\to Map_*(S^2,S^2)$$ (here $j\in\I^{\{1\}}$ is
  represented by $\{1,2\}\in\I$), which obviously does not exist since
  $\Z/2\Z=\pi_1\Omega^4S^4$ does not contain $\Z=\pi_1\Omega^2S^2$ as
  a retract (evaluate the $\Gamma$-spaces on $S^0$).

However, in
section \ref {sec:fundcofseq}, and in subsequent applications, it is
crucial for the equivariant structure of the Loday functor that $\I$ is
defined on $V$.  

In section \ref{subsection:teichmueller} we will give
a very weak extension of  $\pi_0G^A(S^0)$ to cover also the maps $[T\gets
T\times K=T\times K]\in V$ induced by projection onto the first factor.
\end{note}

\subsection{Rectifying $G^A$ and the definition of $\TS(A)$}
We make a tiny recollection on Street's first construction \cite[p. 226]{MR0347936}
and make a
remark about right cofinality, and then we construct the restriction
of the Loday functor
of $A$ to finite sets. 

Let $F \colon J \to \Cat$ be a lax functor. Street's first
construction is a functor
$\widetilde F \colon J \to \Cat$ defined as follows: Given an object
$j$ of $J$, the category $\widetilde F(j)$ has as objects pairs $(\varphi_1,x_1)$
where $\varphi_1 \colon j_1 \to j$ is a morphism in $J$ and $x_1$ is an
object of $F(j_1)$. A morphism from $(\varphi_1,x_1)$ to
$(\varphi_0,x_0)$ consists of a morphism $\psi \colon j_0 \to j_1$
with $\varphi_0 = \varphi_1 \psi $ and a morphism $\alpha \colon
x_1 \to F(\psi)(x_0) $ in $F(j_1)$. The composition of two morphisms
$(\psi_2,\alpha_2) \colon (\varphi_2,x_2) \to (\varphi_1,x_1)$
and
$(\psi_1,\alpha_1) \colon (\varphi_1,x_1) \to (\varphi_0,x_0)$
is the pair $(\psi_2 \psi_1,\beta)$ where $\beta$ is the composition
\begin{displaymath}
  F(\psi_2\psi_1) x_0 \xgets {F(\psi_2,\psi_1)} F(\psi_2)F(\psi_1) x_0
  \xgets {F(\psi_2)\alpha_1} F(\psi_2)x_1 \xgets {\alpha_2} x_2.
\end{displaymath}
If $\psi \colon j \to j'$ is a morphism in $J$, then $\widetilde
F(\psi) \colon \widetilde F(j) \to \widetilde F(j')$ is given by the
formula $\widetilde F(\psi)(\varphi,x) = (\psi \varphi,x)$.

For every object $j$ of $J$ there is a functor $r = r_j \colon \widetilde
F(j) \to F(j)$ with $r_j(\varphi_1,x_1) =
F(\varphi_1)(x_1)$ and with $r_j(\psi,\alpha)$ equal to the composite
\begin{displaymath}
  F(\varphi_1)x_1 \xto{F(\varphi_1)\alpha} F(\varphi_1)F(\psi)x_0
  \xto{F(\varphi_1 \psi)} F(\varphi_1\psi)x_0 = F(\varphi_0) x_0.
\end{displaymath}
These functors assemble to a transformation $r
\colon \widetilde F \to F$.
Note that for every object $x \in F(j)$
the structure morphism $x \to r_j(\id_j,x) = F(\id_j)(x)$ is an object of
the category $x/r_j$. Given another object $\alpha \colon x \to
r_j(\varphi_0,x_0) = F(\varphi_0)(x_0)$
of the category $x/r_j$, the pair $(\varphi_0,\alpha)$ is the only morphism
in $x/r_j$ from $x \to F(\id_j)(x)$ to $\alpha$. Thus the category $x/r_j$ has an
initial element so the functor $r_j$ is
right cofinal.

Now we use Street's first construction to build the Loday functor.
Let $V$ denote the category of spans of finite sets. 
In the notation introduced below
Definition \ref{definepseudo}
we have a lax functor $\I \colon V
\to \Cat$ such that 
$\I(f^*) = (\I L)([Y \xleftarrow f X \xto = X]) \colon \I^Y \to \I^X$ is the functor given  by
precomposing with $f
\colon X \to Y$ and $\I(f_*) = \I([X \xleftarrow = X \xto f Y]) \colon
\I^X \to \I^Y$ is a functor with 
\begin{displaymath}
  \I(f_*)(j)(y) \cong \coprod_{x \in f^{-1}(y)} j(x), 
\end{displaymath}
for $j \in \I^X$. 

\begin{Def}
The functor $\widetilde \I \colon V \to \Cat$ and the 
transformations $r \colon \widetilde \I \to \I$ are defined 
by applying Street's first construction with $J = V$ and $F = \I$.
\end{Def}

Composing with the
  inclusion $Fin\to V$ ({\em pre}composition with functors takes lax
  functors to lax functors) we get a transformation of weak functors
$G^A\circ r\colon\tilde \I \to\gs$:
$$
\xymatrix{
Fin\ar@{=}[r]\ar[d]&Fin\ar[d]\ar@{=}[r]&Fin\ar[dd]^{\gs}\\
V\ar@{=}[r]\ar[d]^{\tilde \I }\ar@{}[dr]|{\overset{r}{\Rightarrow}}&V\ar[d]^{ \I }\ar@{}[r]^{G^A}_{\Rightarrow}&\\
Cat\ar@{=}[r]&Cat\ar@{=}[r]&Cat}
.$$

\begin{Def}
  Given a finite set $S$ and a commutative $\ess$-algebra $A$ the
  {\em Loday functor of $A$ evaluated at $S$} is the $\Gamma$-space $\TS[S](A)$ 
  given by the homotopy colimit
  \begin{displaymath}
    \hocolim[\widetilde  \I (S)] G^A_S \circ r_{S}.
  \end{displaymath}
By the functoriality of the homotopy colimit explained in  Section \ref{funhocolim}, this construction is
functorial in $S\in Fin$, so that we have a functor $\TS(A) \colon \Fin \to
\gs$. 
\end{Def}
Note that we do not use the ring structure on $A$ in the construction
of $\TS[S](A)$ for a fixed finite set $S$, and also note that the action of the
automorphism group of $S$ on $\TS[S](A)$ does not depend on the ring
structure. 

We have introduced the category $V$ in order to have an isomorphism
$\widetilde  \I (S)^{\Group} \cong \widetilde  \I (S/\Group)$ whenever
$\Group$ is a finite group acting on a finite set $S$. This gives us
the following
crucial lemma.  
\begin{lemma}
\label{Bolemma}
  Let $S$ be a finite set and let $\Group$ be a group acting
  on $S$.
  For every functor $Z \colon \I^S \to Spaces$ the map
  \begin{displaymath}
    \hocolim[\widetilde \I(S)] Z \circ r_{S} \to \hocolim[\I^S] Z 
  \end{displaymath}
  induces a weak equivalence on $\Group$-fixed points. 
\end{lemma}
\begin{proof}
  First we treat the case
  where $\Group$ is the trivial group. We have seen that the functor 
  $r_{S} \colon \widetilde \I(S) \to \I^S$ is
  cofinal. Therefore the map
  \begin{displaymath}
    \hocolim[\widetilde \I(S)] Z \circ r_{S} \to \hocolim[\I^S] Z 
  \end{displaymath}
  is an equivalence (right cofinality). 

  In the case where $\Group$ is nontrivial we note that there
  is an isomorphism
  \begin{displaymath}
    (\hocolim[\widetilde  \I (S)] \I \circ r_{S})^{\Group}
    \cong
    \hocolim[\widetilde  \I (S)^{\Group}] \I^{\Group} \circ r_{S}^{\Group}.
  \end{displaymath}
  Further, it is an immediate consequence of Corollary \ref{fixfix} that the
  categories ${\widetilde  \I (S)}^{\Group}$ and $\widetilde  \I (S/\Group)$
  are isomorphic.
  We can now easily reduce to the
  situation where $\Group$ is the trivial group.
\end{proof}

Given an $\ess$-module $A$ and a totally ordered finite set
$T = \{1,\dots,t\}$ we define the $\ess$-module $\bigwedge_{T} A$
inductively by letting $\bigwedge_{T} A = \ess$ for $t = 0$ and
$\bigwedge_{T} A = (\bigwedge_{T \setminus \{t\}} A) \wedge A$ for $t \ge
1$.
\begin{lemma}
\label{lem:Zissmash}
For every finite set $S$ the functor $A \mapsto \dT[S]A$ preserves
connectivity of maps of commutative $\ess$-algebras, and sends stable
equivalences to pointwise equivalences.
If $A$ is cofibrant and $T$ is a finite totally ordered set then there
is a chain of stable equivalences 
between $\dT[T]A$ and 
$\bigwedge_T A$.   
\end{lemma}
\begin{proof}
  The first statement follows from B\"okstedt's Lemma 
(see e.g. \cite[Lemma  2.5.1]{MR2001e:16014})
  since, for given $n$, there is an
  $i\in\I^S$ for which 
  $$Map_*(\bigwedge_{s\in S}S^{i(s)},K \smsh\bigwedge_{s\in
    S}A(S^{i(s)}))=G^A_S(K)(i)\to \TS[S](A)(K)$$ 
  is $n$-connected.  But as $n$ increases and $K$ varies, the term to
  the left is just a concrete model for the (derived) $S$-fold smash of
  $A$ with itself.   

  Let $\widetilde G^A_T\colon \I^T\to\gs$
  be the functor which to the object $i\in\I^T$ assigns the
  $\Gamma$-space
  $$K \mapsto Map_*(\bigwedge_{t\in
    T}S^{i(t)},(\bigwedge_T A)(K \wedge \bigwedge_{t\in T} S^{i(t)})),$$ and with
  structure maps similar to those of $G^A_T$.  
  There are natural maps
  \begin{displaymath} \dT[T]A(K) \gets \hocolim[i \in \I^T] G^A_T(i)(K) \to
  \widetilde {\dT[T]}A(K) := \hocolim[i \in \I^T] \widetilde
  G^A_T(i)(K) \gets (\bigwedge_T A)(K). 
  \end{displaymath}
  The first map is a weak equivalence by Lemma \ref{Bolemma}, the third map
  is always a stable equivalence, and the second map is 
  a stable equivalence if $A$ is cofibrant \cite[Proposition 5.22]{LydGamma}. 
\end{proof}

The smash product of $\ess$-modules is the coproduct in the category
of commutative $\ess$-algebras, and so the category of commutative
$\ess$-algebras is ``tensored'' over totally ordered finite sets through the formula 
$$T \otimes A=\bigwedge_T A.$$ 
Let us choose a total order on every finite set. Using the universal
property of the coproduct we see that $S \mapsto S \otimes A$ is a
functor from the category of finite sets to the category of
commutative $\ess$-algebras. 
The following corollary of the
proof of Lemma \ref{lem:Zissmash} implies that up to homotopy
$\dT[S](A)$ is a commutative $\ess$-algebra.

\begin{cor}
 If $A$ is a commutative $\ess$-algebra which is cofibrant as
 an $\ess$-module, then $\TS[S](A)$ is stably equivalent as an $\ess$-module
 to  the commutative $\ess$-algebra $S \otimes A$
via maps
 that are natural in the finite set $S$.
\end{cor}

\begin{cor}
  Given an $\ess$-module $A$ and finite sets $R$ and $S$ there is an  equivalence
$$\TS[R\times S](A)\simeq \TS[R](\TS[S](A)).$$
\end{cor}

\subsection{Multiplicative structure}
One disadvantage with B\"okstedt's formulation of $THH(A)$ for a
commutative $\ess$-algebra $A$ is that the
multiplicative structure needs some elaboration.  In our model,
$\Lambda_XA$ will automatically be an $\ess$-algebra without any amendments, and it is
appropriately functorial in $X$ and $A$. In particular, our fattening
up of B\"okstedt's model could be viewed as a way of getting the
multiplicative structure in one sweep (though, if one did not care about
the equivariant structure to come, one would drop the complicating
spans, and work with functors from $Fin$ only).  That said, our model is not
strictly commutative, so there is still some advantage to the
categorical smash product constructions in this regard. 

\begin{lemma}\label{Iismonoidal}
  The coproduct in $\I$ gives a 
  transformation with $2$-cells consisting of isomorphisms $\mu\colon (\I\times\I)\circ\Delta\Rightarrow\I$
  of lax functors
  $V\to Cat$, where $\Delta\colon V\to V\times V$ is the diagonal. 
\end{lemma}
\begin{proof}
  We build the $2$-cells consisting of isomorphism from the
  natural isomorphisms we get in  $$\begin{CD}
    \I^X\times\I^X@>>>\I^X\\@VVV@VVV\\\I^Y\times\I^Y@>>>\I^Y
  \end{CD},$$  where the vertical maps are induced by
  $[\xymatrix{X&A\ar[l]_f\ar[r]^g&Y}]\in V$ and the horizontal maps by
  the coproduct in $\I$,
simply (when evaluated on $(i,j)\in\I^X\times\I^X$ and $y\in Y$) by the coherence isomorphism that permutes $\coprod_{a\in
  (Lg)^{-1}(y)}i(Lf(a))\sqcup \coprod_{a\in
  (Lg)^{-1}(y)}j(Lf(a))$ to $\coprod_{a\in
  (Lg)^{-1}(y)}(i(Lf(a))\sqcup j(Lf(a)))$.
\end{proof}

Even more simply, we obtain a natural
  transformation 
$\tilde\mu\colon
  (\tilde\I\times\tilde\I)\circ\Delta\Rightarrow\tilde\I$: if $X\in ob
  V$ then $\tilde\mu_X\colon\tilde\I(X)\times\tilde\I(X)\to\tilde\I(X)$ is given by
$$\left((f_1\colon T_1\to X,i_1\in\I^{T_1}),(f_2\colon T_2\to
  X,i_2\in\I^{T_2})\right)
\quad\mapsto\quad
(f_1+f_2\colon T_1\sqcup T_2\to X, (i_1,i_2)\in\I^{T_1\sqcup T_2})
$$ (here $f_1$ and $f_2$ are maps in $V$, and we have identified
$\I^{T_1\sqcup T_2}$ with $\I^{T_1}\times\I^{T_2}$ for the purpose of
naming elements).  Since we have chosen our category $\I$ of
finite sets and injective functions so that the monoidal structure is strictly associative and unital, we
see that so is $\tilde\mu$.

Note that $X\mapsto\tilde\I(X)\times\tilde\I(X)$ is not the
rectification of $X\mapsto\I(X)\times\I(X)$, so the comparison between
these is not the formal one.

However, there is an invertible modification
$M\colon r\circ\tilde\mu\Rrightarrow\mu\circ(r\times r)$ defined as follows.
If $X\in V$ we let $M_X\colon r_X\circ\tilde\mu_X\Rightarrow
\mu_X\circ(r_X\times r_X)$ be the transformation 
which, when
applied to $\left(([g^1_*f_1^*],i_1)([g^2_*f_2^*],i_2)\right)\in
\tilde\I(X)\times\tilde\I(X)$, is provided by the canonical isomorphism between 
$L([g^1_*f_1^*]+[g^2_*f_2^*])(i_1,i_2)=\left\{x\mapsto \coprod_{(s,x)\in
  L(A_1\sqcup A_2)}(i_1,i_2)(s)\right\}$ and $x\mapsto\left(\coprod_{(t_1,x)\in LA_1}i_1(t_1)\right)\sqcup\left(\coprod_{(t_2,x)\in
  LA_2}i_2(t_2)\right)$.

\begin{lemma}\label{multmod}Let $G^A\smsh G^A$ be the 
transformation
  from $(\I\times\I)\circ\Delta$ to $\gs$ which evaluated on
  $(i,j)\in\I^S\times\I^S$ is given by $G^A_S(i)\smsh G^A_S(j)$.
  Multiplication in $A$ defines a modification $\mu^A$ from $G^A\smsh G^A$ to
  the composite
$$
\xymatrix{
Fin\ar@{=}[r]\ar[d]&Fin\ar[d]\ar@{=}[r]&Fin\ar[dd]^{\gs}\\
V\ar@{=}[r]\ar[d]_{(\I\times \I)\circ
  \Delta}\ar@{}[dr]|-{\overset{}{\Rightarrow}}
&V\ar[d]^{ \I }\ar@{}[r]^{G^A}_{\Rightarrow}&\\
Cat\ar@{=}[r]&Cat\ar@{=}[r]&Cat},
$$
where the transformation in the lower left square comes from Lemma \ref{Iismonoidal}.
\end{lemma}
\begin{proof}
  Explicitly, on objects $i,j\in\I^S$, the natural transformation
  $\mu^A_S(i,j)\colon G^A_S(i)\smsh G_S^A(j)\to G^A_S(i\sqcup j)$ is given by the composite 
$$
  \begin{CD}
    Map_*(\bigwedge_{s\in S}S^{i(s)},\bigwedge_{s\in S}A(S^{i(s)}))
\smsh
Map_*(\bigwedge_{s\in S}S^{j(s)},\bigwedge_{s\in S}A(S^{j(s)}))\\
@VVV\\
Map_*\left((\bigwedge_{s\in S}S^{i(s)})\smsh(\bigwedge_{s\in
  S}S^{j(s)}),(\bigwedge_{s\in S}A(S^{i(s)}))\smsh(\bigwedge_{s\in
  S}A(S^{j(s)}))\right)\\
@VV{\cong}V\\
Map_*\left(\bigwedge_{s\in S}(S^{i(s)}\smsh S^{j(s)}),\bigwedge_{s\in
  S}(A(S^{i(s)})\smsh A(S^{j(s)}))\right)\\
@VVV\\
Map_*\left(\bigwedge_{s\in S}(S^{i(s)\sqcup j(s)}),\bigwedge_{s\in
    S}(A(S^{i(s)\sqcup j(s)})\right)
  \end{CD}
$$
where the first map smashes maps together, the second uses the
canonical rearrangements and the third uses the multiplication in
$A$.  As one checks, $\mu^A_S$ is a natural transformation, since
given inclusions $i\subseteq i'$ and $j\subseteq j'$ in $\I^S$ the
diagram
$$
\begin{CD}
  G^A_S(i)\smsh G_S^A(j)@>{\mu^A_S(i,j)}>> G^A_S(i\sqcup j)\\
  @VVV@VVV\\
  G^A_S(i')\smsh G_S^A(j')@>{\mu^A_S(i',j')}>> G^A_S(i'\sqcup j')
\end{CD}
$$
commutes.

Checking that this defines a modification involves proving
naturality in maps $f\colon S\to T\in Fin$, and basically boils down
to the fact that diagrams like
$$\xymatrix{
(\bigwedge_{s\in S}A(S^{i(s)}))\smsh(\bigwedge_{s\in S}A(S^{j(s)}))
\ar[r]^{\cong}\ar[d]^{\cong}&
{\bigwedge_{s\in S}(A(S^{i(s)})\smsh A(S^{j(s)}))}\ar[d]\\
(\bigwedge_{t\in T}\bigwedge_{s\in
  f^{-1}(t)}A(S^{i(s)}))\smsh(\bigwedge_{t\in T}\bigwedge_{s\in
  f^{-1}(t)}A(S^{j(s)}))\ar[d]&
{\bigwedge_{s\in S}A(S^{i(s)\sqcup j(s)})}\ar[d]^{\cong}\\
(\bigwedge_{t\in T}A(S^{f_*i(t)})\smsh(\bigwedge_{t\in
  T}A(S^{f_*j(t)})\ar[d]^{\cong}&
{\bigwedge_{t\in T}\bigwedge_{s\in f^{-1}(t)}A(S^{i(s)\sqcup j(s)})}\ar[dd]\\
{\bigwedge_{t\in T}A(S^{f_*i(t)}\smsh A(S^{f_*j(t)})}\ar[d]&\\
{\bigwedge_{t\in T}A(S^{f_*i(t)\sqcup f_*j(t)})}\ar[r]^{\cong}&
{\bigwedge_{t\in T}A(S^{f_*(i\sqcup j)(t)})}
}
$$
commute, where the marked isomorphisms are the obvious rearrangements,
and the nonmarked arrows are multiplication.
\end{proof}

\begin{cor}
  Multiplication in $A$ and coherence in $\I$ gives a modification 
$$\xymatrix{
(\tilde\I\times\tilde\I)\circ\Delta\ar@{=>}[r]^{(r\times
      r)}\ar@{=>}[d]^{\tilde\mu}\ar@{}[dr]|{\overset{M}{\Lleftarrow}}&
(\I\times\I)\circ\Delta\ar@{}[drr]|{\overset{\mu^A}{\Lleftarrow}}\ar@{=>}[rr]^{G^A\smsh G^A}\ar@{=>}[d]^{\mu}&&
\gs\ar@{=}[d]\\
\tilde\I\ar@{=>}[r]^r&{\I}\ar@{=>}[rr]_{G^A}&&\gs
}
$$
of 
transformations of lax functors $Fin\to Cat$.
Consequently we get a natural multiplication map
$$\mu^A_S\colon\Lambda_SA\smsh \Lambda_SA\to \Lambda_SA$$
by using the functoriality of homotopy colimits 
$$(\tilde\mu_S,M_S\mu^A_S)\colon\hocolim[\tilde\I(S)\times\tilde\I(S)]G^A_Sr_S\smsh G_S^Ar_S\to
\hocolim[\tilde\I(S)]G_S^Ar_S$$
plus the obvious natural transformation
$$\left(\hocolim[\tilde\I(S)]G_S^Ar_S\right)\smsh\left(\hocolim[\tilde\I(S)]G_S^Ar_S\right)\to
\hocolim[\tilde\I(S)\times\tilde\I(S)]G_S^Ar_S\smsh G_S^Ar_S
$$
\end{cor}

\begin{theo}Let $A$ be a connective commutative $\ess$-algebra.
  The multiplication $\mu^A_S\colon\Lambda_SA\smsh \Lambda_SA\to
  \Lambda_SA$ is associative and unital.
\end{theo}
\begin{proof}
  Proving associativity and unitality is done by direct
  checking using the explicit formula for $\mu^A$ given in the proof
  of Lemma \ref{multmod}, and
  is left as an exercise.  
\end{proof}
\begin{remark}
  Note that this multiplication is also $E_\infty$.  This may be shown, for
  instance, by
  adapting Hesselholt's and Madsen's ``spherewise'' model for the
  multiplication, as in \cite[1.7]{HMfinite}, to the current setup.  This
  approach was
  pointed out to us by C. Schlichtkrull, and details will appear in a
  future publication of Schlichtkrull.  An alternative approach, which
  works if $A$ is cofibrant as an $\ess$-module is to shorten the
  comparison map in 4.3.4 by only using $\tilde\I$'s and observe that
  it is multiplicative, and the right hand side is commutative.
\end{remark}
%
%



\subsection{The Loday functor as a functor of unbased spaces}
If $X$ is a finite (unbased) space, then we define 
$$\TS[X](A)=diag^*\{[q]\mapsto \TS[X_q](A)\}.$$
If $X$ is any (unbased) space, then we define 
$$\TS[X](A)=\colim \TS[S]A$$
where $S$ varies over the finite subspaces of $X$. 

More generally, let $F$ be a functor from the category $Fin$ of finite
(unbased) sets to the category 
of $\ess$-modules ($\Gamma$-spaces).
If $X$ is a finite (unbased) space, then we define 
$F(X) =diag^*\{[q]\mapsto F(X_q)\}.$
If $X$ is any (unbased) space, then we define 
$F(X)=\colim F(S)$
where $S$ varies over the finite subspaces of $X$. 
Note that the following lemma in particular applies to the Loday
functor.  {A map of $\Gamma$-spaces $X\to Y$ is a
  pointwise equivalence (resp. pointwise $n$-connected) if for all finite pointed sets $S$ the map
  $X(S)\to Y(S)$ is a weak equivalence (resp. $n$-connected).  This is stronger than claiming
  that the map of associated spectra is a stable equivalence
  (resp. stably $n$-connected).}
\begin{lemma}
\label{lem:hofunctor}
  Let $f \colon X \to Y$ be a map of simplicial sets and let $f_+
  \colon X_+ \to Y_+$ denote the map obtained by adding a disjoint
  base point.
  \begin{enumerate}
  \item If $f$ is a weak equivalence (resp. $n$-connected), then the induced map 
$$F(f)   \colon F(X)\to F(Y)$$ is a pointwise equivalence (resp. $n$-connected).
  \item If $f$ is a weak equivalence (resp. $n$-connected) after $p$-completion, then 
$F(f)\p$ is  a pointwise equivalence (resp. $n$-connected). 
  \item If $E$ is a spectrum and $E \wedge f_+$ is a stable equivalence, then $E \wedge
    F(f)$ is a stable equivalence.
  \end{enumerate}

\end{lemma}
\begin{proof}
  Let $LF$ denote the functor from simplicial sets to pointed
  simplicial $\Gamma$-spaces with
  \begin{displaymath}
    (LF)(X)_k = \bigvee_{S_0,\dots,S_k} F(S_k) \wedge \Fin(S_k,S_{k-1})_+ \wedge
    \dots \wedge \Fin(S_1,S_0)_+ \wedge \uspaces(S_0,X)_+ 
  \end{displaymath}
The functor $X \mapsto (LF)(X)$ clearly has the stated
properties (for spaces $B$ and $C$ we have equivalences
$(C_+\smsh B_+)\p\simeq (C\p_+\smsh B\p_+)\p$ and
$\uspaces(C,B)\p\simeq\uspaces(C,B\p)$.)  The result now is a
consequence of the fact that there is 
a natural pointwise weak equivalence $(LF)(X) \to F(X)$.
\end{proof}

The above proof also gives the following result.
\begin{lemma}
  \label{prescomp}
  For every space $X$ the map $F(X)\p \to F(X\p)\p$ is a weak
  equivalence.
\end{lemma}

\subsubsection{Adams operations}\label{adams}
Any endomorphism $X\to X$ gives by functoriality rise to an ``operation''
$\TS[X](A)\to\TS[X](A)$.

For instance, if $X=\T^1 \cong \sin|\Sl|$, the $r$-th power map
$\T^1\to\T^1$ gives an operation on ordinary topological Hochschild
homology $THH(A) \simeq \TS[{\T^1}](A)$ corresponding to the one described  by McCarthy
\cite{McCoper}.  There is a difference in description in that he uses
the usual simplicial circle $\Sl$, and so has to subdivide to express
the operation, whereas in our situation this can be viewed as a
reflection of the fact that (the image of the acyclic cofibration)
$sd_r\Sl\subseteq\sin|sd_r\Sl|\cong\T^1$ is one of the
legitimate finite simplicial subsets over which we perform our colimit.

This is general: for any functor $F$ from $Fin$ to, say, spaces, the
$r$th Adams operation on $\pi_q(F(\Sl))$ is given through McCarthy's
interpretation as the composite
$$\pi_q(F(\Sl))\cong\pi_q(F(sd_r(\Sl)))\to \pi_q(F(\Sl))$$ where the
last map is induced by a certain map $sd_r\Sl\to \Sl$.  That this
corresponds to our definition is seen through the 
commutativity of
$$\xymatrix{\sin|sd_r\Sl|\ar[r]^\cong&\sin|\Sl|\ar[r]^{z\mapsto z^r}&\sin|\Sl|\\
sd_r\Sl\ar[u]^\simeq\ar[rr]&& \Sl\ar[u]^\simeq}
$$

For higher dimensional tori one gets
operations for every integral matrix $\alpha$ with nonzero determinant.
In this case Lemma \ref{prescomp} gives  $\TS[\T^n](A)_p\simeq \TS[\T^n_p](A)_p$ and so
we have an action by $GL_n(\bZ_p)$ which in the one-dimensional case
corresponds to the action by the $p$-adic units. 

\section{The fundamental cofibration sequence}
\label{sec:fundcofseq}

In this section we investigate fixed points of the Loday functor $\dT[X]A$
with respect to 
group actions on $X$. This leads to a generalized version of the
fundamental cofibration sequence for B\"okstedt's topological
Hochschild homology. On the way we describe the norm map from homotopy
orbits to homotopy fixed points.

\subsection{The norm cofibration sequence}
\label{sec:herecomesnorm}
Let $G$ be a finite group.
For convenience, we introduce the following shorthand notation: if $S$
is a finite set and $A$ is a
commutative $\ess$-algebra and $j\in\I^{S}$ we
write $A(j)$ for the space
$$\bigwedge_{s\in
  S}A(S^{j(s)}),$$
so that (c.f. \ref{thelont}) $G^A_S(j)=Map_*(\ess(j),A(j)).$

We say that a set $\mathcal F$ of subgroups of $G$
is a {\em closed family} if it has the property  that if $H \in \mathcal F$
and $H \le gKg^{-1}$, for $g \in G$ and a subgroup $K$ of $G$, then $K
\in \mathcal F$.
If $G$ acts on $S$ and $\mathcal F$ is a closed family of subgroups in $G$,
 we define
\begin{displaymath}
  G^A_S(\mathcal F)(j) = Map_*(\bigcup_{H \in \mathcal F} S(j)^H,A(j)).
\end{displaymath}
This is a transformation from the  
weak functor $S\mapsto\I^S$ to the constant functor from finite
$G$-sets to categories sending everything to $\gs$.
Just like when we constructed $\dT(A)$, we can define a functor $\dT[]^A(\mathcal F)$ from the
category of $G$-spaces to the category of spectra, with
\begin{displaymath}
  \dT[S]^A(\mathcal F) = \hocolim[\Kd\I(S)]G^A_S(\mathcal F) \circ r_S
\end{displaymath}
when $S$ is a finite $G$-set. 
In particular, if $\mathcal F$ is the empty family of subgroups of $G$,
then $\dT[S]^A(\mathcal F) = *$.
\begin{lemma}
\label{resnormal}
  Let $N$ be a normal subgroup in $G$ and let $\mathcal F$ be the
  closed family of subgroups of $G$ consisting of
  the subgroups containing $N$. There is, for every $j \in \I^S$, a natural
  isomorphism of $G/N$-spaces of 
  the form 
\begin{displaymath}
  G^A_S(\mathcal F)(j)^N = Map_*(S(j)^N,A(j)^N).
\end{displaymath}
In particular, if $\mathcal F$ is the family of all subgroups of $G$,
then $\dT[S]^A(\mathcal F) 
= \dT[S](A)$. 
\end{lemma}

Note that if $\mathcal G \subseteq \mathcal F$ is an inclusion of closed
families of subgroups of $G$, then the
inclusion 
\begin{displaymath}
  \bigcup_{H \in \mathcal G} S(j)^H \subseteq  \bigcup_{H \in \mathcal F} S(j)^H 
\end{displaymath}
induces a $G$-map
\begin{displaymath}
  \mathrm{res} : \dT[S]^A(\mathcal F) \to \dT[S]^A(\mathcal G). 
\end{displaymath}
We refer to these maps as {\em restriction maps}. However they should not be
confused with the restriction maps $R$ in Section 
\ref{sec:resmap} generalizing the restriction maps of Hesselholt and
Madsen \cite{HM1}. If $\mathcal F$ is the family of all subgroups in $G$,
then by Lemma \ref{resnormal} the restriction map takes the form
$$\mathrm{res} \colon \dT[S](A) \to \dT[S]^A(\mathcal G).$$

If the complement of $\mathcal G$ in $\mathcal F$ is equal to the
conjugacy class of a subgroup $K$ in $G$ we shall say that $\mathcal
F$ and $\mathcal G$ are {\em $K$-adjacent}. Moreover we let $W_G K =
N_GK/K$ denote the {\em Weyl group} of the subgroup $K$ in $G$. Here $N_GK$
is the {\em normalizer} of $K$ in $G$ consisting of those $g$ in $G$ with $gKg^{-1} = K$. 

In the next lemma we shall describe the homotopy fibre of the
restriction map when $\mathcal G  \subseteq \mathcal F$ are
$K$-adjacent. In the proof we will need a norm map $Z_{hG} \to Z^{hG}$
for $Z$ a $\Gamma$-space with action of $G$. Later we shall relate
the norm map to a transfer map, and therefore we need to choose a
specific representative of it. The discussion below is modeled on
Weiss and William's paper \cite[Section 2]{WW}. Given $n \ge 0$ we let $S^{nG}$ denote the
$G$-fold smash product of the $n$-sphere. This is our model for the
one-point compactification of the regular representation of $G$. The
map $\alpha \colon G_+ \wedge S^{nG} \to Map_*(G_+,S^{nG})$, with
$\alpha(g,x)(g) = x$ and with $\alpha(g,x)(h)$ equal to the base point in
$S^{nG}$ if $h \ne g$, and the diagonal inclusion $S^{nG} \to
Map_*(G_+,S^{nG})$ induce $G$-maps
\begin{eqnarray*}
  Map_*(S^{nG},Z(S^{nG})) 
  &\cong& 
  Map_*(G_+ \wedge S^{nG},Z(S^{nG}))^G \\ 
  &\gets& 
  Map_*(Map_*(G_+,S^{nG}),Z(S^{nG}))^G \\ 
  &\to& 
  Map_*(S^{nG},Z(S^{nG}))^G. 
\end{eqnarray*}
Here $G$ acts by conjugation on the first space, through the left
action of $G$ on itself on the second space and through the right
action of $G$ on itself on the third space. On the last
space $G$ acts trivially.
Passing to homotopy colimits we obtain a weak $G$-map
\begin{eqnarray*}
  \hocolim[n] Map_*(S^{nG},Z(S^{nG})) 
  &\xleftarrow{\simeq}& 
  \hocolim[n]  Map_*(Map_*(G_+,S^{nG}),Z(S^{nG}))^G \\ 
  &\to& 
  \hocolim[n] Map_*(S^{nG},Z(S^{nG}))^G 
\end{eqnarray*}
of $\Gamma$-spaces. This is an additive transfer associated to
$G$. Observing that there is a chain of stable equivalences between
$Z$ and $\hocolim[n] Map_*(S^{nG},Z(S^{nG}))$ we denote the resulting
map in the homotopy category of (na\"ive) $G$-$\Gamma$-spaces by 
$V^G \colon Z \to \hocolim[n] Map_*(S^{nG},Z(S^{nG}))^G$. The homotopy
category we have in mind is the one where we invert the $G$-maps whose
underlying non-equivariant maps are stable equivalences. 
Since $G$ acts trivially on the target, there is an induced map $\widehat V^G \colon
Z_{hG} \to \hocolim[n] Map_*(S^{nG},Z(S^{nG}))^G$.
The {\em norm map} is the weak map $N \colon Z_{hG} \to Z^{hG}$ given
by the composition
\begin{eqnarray*}
  Z_{hG} &\to& (Map_*(EG_+,Z))_{hG} \\ 
  &\xrightarrow {\widehat V^G}&
  \hocolim[n] Map_*(S^{nG},Map_*(EG_+,Z(S^{nG})))^G \simeq Z^{hG}   
\end{eqnarray*}
It is an easy matter to check that the composite homomorphism
\begin{displaymath}
  \pi_* Z \to \pi_* Z_{hG} \xto N \pi_* Z^{hG} \to \pi_* Z
\end{displaymath}
is multiplication by $\sum_{g \in G} g \in \pi_0(G_+ \wedge \ess)$. If
$Y$ is a $\Gamma$-space with trivial action of $G$, then there is a
commutative diagram of the form
\begin{displaymath}
  \begin{CD}
    \pi_*(G_+ \wedge Y) @>>>
    \pi_*(G_+ \wedge Y)_{hG} @>N>>
    \pi_*(G_+ \wedge Y)^{hG} @>>>
    \pi_*(G_+ \wedge Y) \\
    @V{\cong}VV @A{\cong}AA @V{\cong}VV @A{\cong}AA \\
    \bigoplus_G \pi_*Y @>{\nabla}>> \pi_*Y @. \pi_*Y @>{\Delta}>> \bigoplus_G \pi_*Y
  \end{CD},
\end{displaymath}
where $\nabla(\{g \mapsto y_g\}) = \sum_g y_g$ and $\Delta(y) = \{g
\mapsto y\}$. Since the composition $\bigoplus_G \pi_*Y \to
\bigoplus_G \pi_*Y$ from the lower left corner to the lower right
corner is the norm map, the composition $\pi_*Y \to \pi_*Y$ must be the
identity on $\pi_*Y$. Thus the norm map $N \colon (G_+ \wedge Y)_{hG}
\to (G_+ \wedge Y)^{hG}$ is a stable equivalence. An easy inductive
argument on cells shows that the norm map $N \colon Z_{hG} \to Z^{hG}$
is a weak equivalence when $Z$ is of the form $Z = Map_*(U,X)$ for $X$
a $G$-$\Gamma$-space and $U$ a finite free $G$-space.

\begin{lemma}[the norm cofiber sequence]
\label{lem:fundcofseq1}
  Let $G$ be a finite group, let  $\mathcal G \subseteq \mathcal F$
  be $K$-adjacent families of subgroups of $G$ and let $X$ be a non-empty
  free $G$-space.
  For every 
  commutative $\ess$-algebra $A$ the 
  homotopy fiber of the restriction map
  \begin{displaymath}
    \mathrm{res} \colon [\dT[X]^A({\mathcal F})]^G \to [\dT[X]^A({\mathcal G})]^G  
  \end{displaymath}
is equivalent to the homotopy orbit spectrum
\begin{displaymath}
  [\dT[{X/K}]A]_{hW_GK}.
\end{displaymath}
\end{lemma}
\begin{proof}
  Since we work stably it suffices to prove the result for every
  discrete set $X = S$.
  Moreover, since fixed points, finite homotopy limits and
  homotopy orbits commute with filtered colimits, it suffices to
  consider the case where $S$ is a finite set.
  Let $(\phi,i)\in [\Kd\I(S)]^G$, write $j=\I({\phi})i\in\I^{S}$  
and consider the restriction map
 $$
Map_*\left(\bigcup_{H\in \mathcal F}\left(\ess(j)\right)^H,A(j)\right)^G \to
Map_*\left(\bigcup_{H \in \mathcal G}\left(\ess(j)\right)^H,A(j)\right)^G.
  $$
The fiber of this fibration is 
$$Map_*\left(Z(j),A(j)\right)^G
$$ 
where 
$$Z(j)=\left(\bigcup_{H\in \mathcal F}\left(\ess(j)\right)^H\right) \left/
\left(\bigcup_{H\in \mathcal G}\left(\ess(j)\right)^H\right) \right. .$$
Apart from the base point, the $G$-orbits of $Z(j)$ are all isomorphic to $G/K$.
Therefore we have isomorphisms
\begin{eqnarray*}
  Map_*\left(Z(j),A(j)\right)^G 
  &\cong& Map_*\left(Z(j)^K,A(j)\right)^{N_G
  K} \\
  &\cong& Map_*\left(Z(j)^K,A(j)^K\right)^{W_G
  K}. 
\end{eqnarray*}
Since $Z(j)^K$ is a free $W_GK$-space, the map 
\begin{displaymath}
  Map_*\left(Z(j)^K,A(j)^K\right)^{hW_G
  K} \to  Map_*\left(Z(j)^K,A(j)^K\right)^{W_G
  K} 
\end{displaymath}
is an equivalence. 
Since $W_GK$ is a finite group and $Z(j)^K$ is a finite
$W_GK$-space, the norm map 
\begin{displaymath}
  N \colon Map_*(Z(j)^K,A(j)^K)_{hW_G K}
 \to  Map_*\left(Z(j)^K,A(j)^K\right)^{hW_G
  K}
\end{displaymath}
is a stable equivalence, so we can
deduce that 
$$\hocolim[{(\phi,i)\in
    [\Kd\I(S)]^G}] Map_*\left(Z(\I(\phi)i),A(\I(\phi)i)\right)^G
$$ 
is equivalent
to 
$$
\hocolim[{(\phi,i)\in
    [\Kd\I(S)]^G}] Map_*(Z(\I(\phi)i)^K,B(\I(\phi)i)^K)_{hW_GK}.
$$

Note that $H \in \mathcal G$ implies that $K$ is a proper subgroup
of the subgroup $H\cdot K$ of $G$ generated by $H$ and $K$. Since $K$
acts freely on $S$, the space
\begin{displaymath}
  \left[ \bigcup_{H \in \mathcal G} \ess(j)^{H} \right]^K =
  \bigcup_{H \in \mathcal G} \ess(j)^{H\cdot K}
\end{displaymath}
is at most $\sum_{s \in S} j(s)/2|K|$-dimensional and $S^k \wedge
A(j)^K$ is $k -1 + \sum_{s \in S} j(s)/|K|$-connected. Using that
$K$-fixed points commute with quotients by sub $K$-spaces we can conclude that
the map
$$
Map_*\left(Z(j)^K,S^k\smsh A(j)^K\right)_{hW_GK}\to 
Map_*\left(\ess(j)^K,S^k\smsh A(j)^K\right)_{hW_GK}$$
is $k-1+\sum_{s\in S}j(s)/2|K|$-connected.

These considerations are functorial in $(\phi,i)$, and since $X$ is
non-empty, together with Lemma
\ref{Bolemma} and B\"okstedt's Lemma 
(see e.g. \cite[Lemma
  2.5.1]{MR2001e:16014})
they give that
the homotopy fiber 
of the map
$[Z_S^A({\mathcal F})]^G \to [Z_S^A({\mathcal G})]^G  $
 is equivalent to
 \begin{eqnarray*}
\lefteqn{\hocolim[{(\phi,i)\in
    [\Kd\I(S)]^G}] Map_*\left(Z(\I(\phi)i),A(\I(\phi)i)\right)^G} \\
   &\simeq& 
\hocolim[{(\phi,i)\in
    [\Kd\I(S)]^G}] Map_*(\ess(\I(\phi)i)^K,A(\I(\phi)i)^K)_{hW_GK}.
 \end{eqnarray*}
The $G$-maps 
\begin{eqnarray*}
  \lefteqn{\hocolim[{(\phi,i)\in
  [\Kd\I(S)]^G}]Map_*\left(\ess(\I(\phi)i)^K,
  A(\I(\phi)i)^K\right)} \\
  &\to&
  \hocolim[{(\phi,i)\in
  [\Kd\I(S)]^K}]Map_*\left(\ess(\I(\phi)i)^K,A(\I(\phi)i)^K\right)
  \\
  &\to&
  \hocolim[{(\phi,i)\in
  \Kd\I(S/K)}]Map_*\left(\ess(\I(\phi)i),A(\I(\phi)i)\right)  
\end{eqnarray*} 
induced by the inclusion $[\Kd\I(S)]^G\subseteq
[\Kd\I(S)]^K$ and the isomorphism $[\Kd\I(S)]^K \cong \Kd\I(S/K)$
are by Lemma \ref{Bolemma} equivalences, and hence we get that the
homotopy fiber is equivalent to  
\begin{displaymath}
  [Z_{S/K}A]_{hW_GK}.
\end{displaymath}
\end{proof}

\begin{cor}\label{cor:httype}
   Let $G$ be a finite group and let $X$ be a non-empty free
   $G$-space. For every closed family $\mathcal F$ of subgroups of 
   $G$ the functor
   $A \mapsto [\dT[X]^A(\mathcal F)]^G$ preserves connectivity of maps of commutative $\ess$-algebras.
and has values in very special
   $\Gamma$-spaces.  Furthermore we have a natural equivalence $([\dT[X]^A(\mathcal F)]^G)\p\simeq
   [\dT[X]^A(\mathcal F)\p]^G$. 
\end{cor}
\begin{proof}
  We make induction on the partially ordered set of closed
  families ordered by inclusion.
  If $\mathcal F$ is empty, then $[\dT[X](\mathcal F)]^G$ is
  contractible, and thus the result holds. 
  Otherwise we may choose a minimal
  subgroup $K$ in $\mathcal F$ and a closed family $\mathcal G
  \subseteq \mathcal F$ of
  subgroups of $G$ so that $\mathcal G$ is $K$-adjacent to $\mathcal
  F$. By Lemma \ref{lem:fundcofseq1}
  there is a cofibration sequence
  \begin{displaymath}
    [\dT[X/K](A)]_{hW_GK} \to [\dT[X]^A(\mathcal F)]^G\to [\dT[X]^A(\mathcal G)]^G.
  \end{displaymath}
  By Lemma \ref{lem:Zissmash} the functor $A \mapsto [\dT[X/K](A)]$ preserves
 connectivity,
and thus also the functor 
  $A \mapsto
  [\dT[X/K](A)]_{hW_GK}$ preserves connectivity. Together with the inductive
  assumption and the five lemma, this implies that the functor $A \mapsto
  [\dT[X]^A(\mathcal F)]^G$ preserves connectivity.  Since completion
commutes with homotopy orbits the 
second statement is proved in the same way. 
\end{proof}

\begin{cor}\label{cor:R0conn}
  Let $\mathcal G \subseteq \mathcal F$ be closed families of
  subgroups of a finite group $G$. For every commutative $\ess$-algebra $A$ and non-empty free
  $G$-space $X$ the restriction map $[\dT[X]^A(\mathcal F)]^G \to
  [\dT[X]^A(\mathcal G)]^G$ is $0$-connected.
\end{cor}
\begin{proof}
  Choose adjacent closed families $\mathcal G_1 \subseteq \dots
  \subseteq \mathcal G_n$ of subgroups of $G$ with $\mathcal G_1 =
  \mathcal G$ and $\mathcal G_n =
  \mathcal F$ and apply Lemma \ref{lem:fundcofseq1}
\end{proof}

\subsection{The restriction map}
\label{sec:resmap}

Given a normal subgroup $N$ of $G$ we
let $\diag^N_S \colon S\to S/N$ be the projection and we let $\mathcal
F$ denote the closed family consisting of all subgroups of $G$
containing $N$.
Notice that Corollary \ref{fixfix} implies that the induced functor
$(\diag^N_S)^* \colon \Kd \I(S/N) \to \Kd \I(S)$ induces an isomorphism
$(\diag^N_S)^* \colon \Kd \I(S/N) \xto \cong \Kd \I(S)^N$.
\begin{lemma}
  Let $N$ be a normal subgroup of a finite group $G$ and let $S$ be a
  finite $G$-set. For every commutative $\ess$-algebra $A$ the
  $G/N$-spaces $[\dT[S]^A(\mathcal F)]^N$ and 
  $$\hocolim[\Kd\I({S/N})]\left[G^A_S(\mathcal F) \circ
  r_S \circ (\diag^N_S)^* \right]^N$$ are isomorphic.
\end{lemma}
\begin{proof}
  The $G/N$-space $[\dT[S]^A(\mathcal F)]^N$ is defined to be the
  $N$-fix points of the
  homotopy colimit 
  $\hocolim[\Kd\I(S)]G^A_S(\mathcal F) \circ
    r_S.$ 
  Homotopy colimits and fixed points commute in the
  sense that this $N$-fix point space is isomorphic to the homotopy
  colimit $\hocolim[\Kd\I({S})^N]\left[G^A_S(\mathcal F)\circ
  r_S\right]^N$. The result now follows from the fact that
  $(\diag^N_S)^* \colon \Kd \I(S/N) \xto \cong \Kd \I(S)^N$ is an isomorphism.
\end{proof}

Let $(\phi,i) \in \widetilde F(S/N)$ and let $j = \I((\diag^N_S)^*
\phi)(i)$. The isomorphisms $S(j)^N \cong S(I(\phi)(i))$ and $A(j)^N
\cong A(I(\phi)(i))$ induce a
natural isomorphism of the
form
\begin{eqnarray*}
  [(G_S^A(\mathcal F) \circ r_S \circ (\diag^N_S)^*)(\phi,i)]^G &=&
  Map_*(\cup_{N \le H} S(j)^H,A(j))^G \\
  &=& Map_*(\cup_{N \le H} S(j)^H,\cup_{N \le H} A(j)^H)^G \\
  &=& Map_*(S(j)^N,A(j)^N)^{G/N} \\
  &\cong& [(G^A_{S/N} \circ r_{S/N})(\phi,i)]^{G/N}.
\end{eqnarray*}
Thus the inclusion of $\mathcal F$ in the family of all subgroups of
$G$ induces a natural
transformation of functors of the form
\begin{equation}
  \label{eq:label}
  \left[G^A_S\circ r_S \circ (\diag^N_S)^*\right]^G \Rightarrow
  \left[G^A_S(\mathcal F) \circ r_S \circ (\diag^N_S)^*\right]^G \cong
  \left[G^A_{S/N}\circ r_{S/N} \right]^{G/N}  
\end{equation}
of functors $\Kd\I({S/N})\to \gs$
given on $(\phi,i)$ by restricting to $N$-fixed points. 
This induces a modification, and hence a natural map
\begin{displaymath}
  {R^G_N \colon  \left[\hocolim[\Kd\I(S)][G^A_S \circ
    r_S]\right]^G \to} 
\left[\hocolim[\Kd\I({S/N})]G^A_{S/N}\circ r_{S/N}
  \right]^{G/N},
\end{displaymath}
that is, a map
\begin{displaymath}
  R^G_N \colon [\dT[S](A)]^G  \to [\dT[S/N](A)]^{G/N}.
\end{displaymath}

If $G$ is a finite group and $X$ is a simplicial $G$-set, let
$\mathcal F^X_G$ be the filtered category 
of finite $G$-subspaces of $X$ and inclusions.
Notice that $\mathcal F^X_G\subseteq F^X_{\{1\}}$ is right cofinal,
and so colimits over $F^X_G$ and $F^X_{\{1\}}$ are isomorphic, and the
isomorphism is $G$-equivariant. 

  \begin{Def}
  Let $A$ be a commutative $\ess$-algebra, $G$ a finite group acting
  on a simplicial set $X$ and $N$ a normal subgroup of $G$.
  The {\em restriction map} 
  \begin{displaymath}
    R^G_{N} \colon \left[\dT[X](A)\right]^G \to \left[\dT[X/N](A)\right]^{G/N}
  \end{displaymath}
  is the composite
  \begin{eqnarray*}
    \left[\dT[X](A)\right]^G &\cong& \colim[S \in \mathcal F^X_G]
    [\dT[S](A)]^G \\
    &\to& \colim[S \in \mathcal F^{X}_{G}]
    [\dT[{S/N}](A)]^{G/N} \cong \colim[U \in \mathcal F^{X/N}_{G/N}]
    [\dT[{U}](A)]^{G/N} \cong \left[\dT[X/N](A)\right]^{G/N}.
  \end{eqnarray*}
  \end{Def}
The restriction map $R^G_N$ is natural in the commutative $\ess$-algebra
$A$ and in the $G$-space $X$. Moreover Corollary \ref{cor:R0conn}
implies that
$R^G_N$ is $0$-connected.
The following lemma is direct from the definition, but is important
for future reference:
\begin{lemma}\label{lem:RF=FR}
  Let $A$ be a commutative $\ess$-algebra, $X$ a $G$-space with $G$ a
  finite group, and let
  $H\subseteq N\subseteq G$ be normal subgroups.  Then
  $$R^G_N=R^{G/H}_{N/H}R^G_{H}.$$  If $F^G_N\colon
  [\dT[X](A)]^G\subseteq [\dT[X](A)]^N$ denotes the inclusion of fixed
  points, then we also have that 
$$
  \begin{CD}
    [\dT[X](A)]^G@>{R^G_H}>>[\dT[X/H](A)]^{G/H}\\
    @V{F^G_N}VV@V{F^{G/H}_{N/H}}VV\\
    [\dT[X](A)]^N@>{R^N_H}>>[\dT[X/H](A)]^{N/H}
  \end{CD}
$$
commutes.
\end{lemma}

Most importantly for our applications, we have the {\em fundamental cofiber sequence}:

\begin{lemma}\label{lem:funcofseq}
  Let $G$ be a finite abelian group,
  $X$ a non-empty  free $G$-space and $A$ a commutative $\ess$-algebra. The
  homotopy fibre of the map
  \begin{displaymath}
    [\dT[X](A)]^G \to \holim[\trivgroup \ne H \le G] [\dT[{X/H}](A)]^{G/H}
  \end{displaymath}
  induced by the restriction maps is equivalent to the homotopy orbit
  spectrum $[\dT[S](A)]_{hG}$.
\end{lemma}
\begin{proof}
  Let $\mathcal F = \{H \colon \trivgroup \ne H \le G\}$. In order to apply Lemma \ref{lem:fundcofseq1} it suffices 
to show that the space
$$
[\dT[X](\mathcal F)]^G \cong \hocolim[{(\phi,i)\in [\Kd\I(S)]^G}]Map_*\left[\bigcup_{\trivgroup \neq H\le G}\left(\ess(\I(\phi)i)\right)^H,A(\I(\phi)i)\right]^G
$$
is equivalent to
$$[\holim[\trivgroup \neq H\le G]\dT[{S/H}](A)]^{G} = \holim[\trivgroup \ne H \le G] [\dT[{X/H}](A)]^{G/H}.$$

First we notice that the natural map 
\begin{multline*}
\hocolim[{(\phi,i)\in [\Kd\I(S)]^G}]Map_*\left[\bigcup_{\trivgroup \neq H\le G}\left(\ess(\I(\phi)i)\right)^H,A(\I(\phi)i)\right]^G\\
\to 
\left[ \holim[\trivgroup \neq H\le G]\hocolim[{(\phi,i)\in [\Kd\I(S)]^G}]Map_*\left(\left(\ess(\I(\phi)i)\right)^H,A(\I(\phi)i)\right)\right]^G\\
\cong
\left[\holim[\trivgroup \neq H\le G]\hocolim[{(\phi,i)\in [\Kd\I(S)]^G}]Map_*\left(\left(\ess(\I(\phi)i)\right)^H,\left(A(\I(\phi)i)\right)^H\right)\right]^{G}
\end{multline*}
is an equivalence (again by Lemma \ref{Bolemma}, and since $G$ is
finite).  The last term is isomorphic to 
$$[\holim[\trivgroup \neq H\le G]\hocolim[{[\Kd\I({S/G})]}]G^A_{S/H} \circ
  r_{S/H} \circ (\diag^G_H)^*]^G$$ 
which again is isomorphic to 
$$[\holim[\trivgroup \neq H\le G]\dT[{S/H}](A)]^G.$$
That the induced maps are the restriction maps is now clear.
\end{proof}

\setcounter{subsection}{0}\setcounter{subsubsection}{0}
\section{The Burnside-Witt construction}
Hesselholt and Madsen prove in \cite{HM1} that if $A$ is a
discrete commutative ring and $TR(A)$ is the homotopy inverse limit
over the restriction maps between the fixed points of the (one
dimensional) topological Hochschild homology under finite subgroups of
the circle, then $\pi_0TR(A)$ is isomorphic to the ring of Witt
vectors over $A$.  

In this section we prove an analogous result for arbitrary finite
groups. Let $G$ be a finite group, let $X$ be a connected
free $G$-space and let $A$ be a commutative $\ess$-algebra $A$. There is a
canonical isomorphism of the form
\begin{displaymath}
  \pi_0 \dT[X](A)^G \cong \Witt_G(\pi_0A) ,
\end{displaymath}
 where the latter ring is Dress and
Siebeneicher's ``Burnside-Witt''-ring  \cite{DressSieb} of the
commutative ring $\pi_0A$ over
the group
$G$.

\subsection{The Burnside-Witt ring}
\label{subseq:BWring}

We shall review some elementary facts about the $G$-typical
Burnside-Witt ring $\Witt_G(B)$ of a commutative ring $B$.
For more details on the Burnside-Witt construction, the reader may consult
for instance \cite{gray} and \cite{bruntambara} in addition to \cite{DressSieb}. 
The underlying set of $\Witt_G(B)$ is the set
\begin{displaymath}
  \left[ \prod_{H \le G} B \right]^G
\end{displaymath}
where $G$ acts on the product by taking the factor of $B$ corresponding to a subgroup
$H$ of $G$ identically to the factor corresponding to $gHg^{-1}$. For every
subgroup $K \le G$ there is a ring homomorphism $\phi_K \colon
\Witt_G(B) \to B$ taking $x = (x_H)_{H \le G}$ to
\begin{displaymath}
  \phi_K(x) = \sum_{[H]} |(G/H)^K| \, x_H^{|H|/|K|},
\end{displaymath}
where the sum runs over the conjugacy classes of subgroups of $G$. 
The
endofunctor $B \mapsto \Witt_G(B)$ on the category of commutative
rings is uniquely determined by the following facts:
Firstly, the underlying set  of $\Witt_G(B) = \left[ \prod B \right]^G$ is
functorial in $B$ and secondly, $\phi_K$ is a natural ring
homomorphism for every subgroup $K$ of $G$. Dress and Siebeneicher
establish the existence of the ring structure on $\Witt_G(B)$ by
making a detailed study of the combinatorics of finite $G$-sets.
Below we use the functor $\dT$ to give a different proof of the
existence of the ring structure 
on $\Witt_G(B)$.

Note that when the underlying
additive group of $B$ is torsion free
the map
\begin{displaymath}
  \phi \colon \Witt_G(B) \to \left[\prod_{H \le G} B \right]^G
\end{displaymath}
with $\phi(x)_K = \phi_K(x)$ is injective.
For every closed family $\mathcal F$ of subgroups of $G$ 
we define $\Witt_{\mathcal F}(B)$ to be the set 
$$\Witt_{\mathcal F}(B) = \left[ \prod_{H \in
    \mathcal F} B\right]^G.$$

\subsection{The Teichm\"uller map}
\label{subsection:teichmueller}

We shall show by induction on the size of
$\mathcal F$, that for every connected free $G$-space
$X$ there is a bijection  $\Witt_{\mathcal F}(\pi_0A)
\to \pi_0 \dT[X]^A(\mathcal F)^G$.  In order to do this we need to
recollect some structure on 
$\dT[X]^A(\mathcal F)^G$. 

Suppose that $S$ is a finite free $G$-set, let $K$ be a subgroup of
$G$, let $j \in 
[\I^S]^G$ and let $L$ be a pointed space. Taking homotopy colimits over the maps
\begin{eqnarray*}
  Map_*(\bigcup_{H \in \mathcal F} S(j)^H, A(j) \wedge L)^K &\cong&
  Map_*(G_+ \wedge_K \bigcup_{H \in \mathcal F} S(j)^H, A(j) \wedge
  L)^G \\
  &\gets&
  Map_*(Map_*(G_+,\bigcup_{H \in \mathcal F} S(j)^H)^K, A(j) \wedge
  L)^G \\
  &\to &
  Map_*(\bigcup_{H \in \mathcal F} S(j)^H,A(j) \wedge L)^G
\end{eqnarray*}
we obtain a weak map $V_K^G$ of the form $\dT[S]^A(\mathcal F)^H \xleftarrow{\simeq}
\widetilde {\dT[S]^A}(\mathcal F)^G \to
\dT[S](A)^G$, where $\widetilde {\dT[S]^A}(\mathcal F)^G(L)$ is the
homotopy colimit of the middle term in the diagram displayed
above. Extending from finite sets $S$ to $G$-spaces $X$ we denote 
the induced homomorphism on $\pi_0$ by 
\begin{displaymath}
  V^G_K \colon \pi_0 \dT[X]^A(\mathcal F)^K \to
  \pi_0 \dT[X]^A(\mathcal F)^G,
\end{displaymath}
and call it the {\em additive transfer}.
Clearly, if $\mathcal G \subseteq \mathcal F$ is an inclusion of
closed families of subgroups of $G$, then the restriction maps defined
in \ref{sec:herecomesnorm} commute with additive transfers:
\begin{displaymath}
  \begin{CD}
  \pi_0 \dT[X]^A(\mathcal F)^K @>{V^G_K}>>
  \pi_0 \dT[X]^A(\mathcal F)^G \\
  @V{\mathrm {res}}VV @V{\mathrm{res}}VV \\
  \pi_0 \dT[X]^A(\mathcal G)^K @>{V^G_K}>>
  \pi_0 \dT[X]^A(\mathcal G)^G    
  \end{CD}.
\end{displaymath}
The homomorphism $V^G_K$ is the additive transfer associated to the
inclusion $K \le G$. 
We also need a kind of multiplicative transfer.
Given a pointed space $L$, we can consider the $K$-fold smash power
$L^{\wedge K}$, where $K$ acts by permuting the smash-factors. Let $S$
be a finite $G$-set, let $j \in \I^S$ and let $p^*j \in [\I^{K
  \times S}]^K$ correspond to $j$ under the isomorphism $\I^S \cong [\I^{K
  \times S}]^K$ induced by the projection $p \colon K \times S \to S$. The map
\begin{displaymath}
  Map_*(S(j),A(j) \wedge L)^{\wedge K}
  \to
  Map_*(S(p^*j),A(p^*j) \wedge L^{\wedge K})
\end{displaymath}
taking a map $\alpha$ to its $K$-th smash power $\alpha^{\wedge K}$
induces a map 
\begin{displaymath}
  [\dT[X](A)(L)^{\wedge K}]^K \to [\dT[K \times
  X](A)(L^{\wedge K})]^K.
\end{displaymath}
Thus there is a map
\begin{eqnarray*}
  \dT[X](A)(L) {\cong} [\dT[X](A)(L)^{\wedge K}]^K 
  \to [\dT[K \times
  X](A)(L^{\wedge K})]^K 
  \xto{\dT[p](A)} [\dT[X](A)(L^{\wedge K})]^K.
\end{eqnarray*}
When $L = S^0$ we denote the induced map on $\pi_0$ by
\begin{displaymath}
  \Delta_K \colon \pi_0 \dT[X](A) \to \pi_0 \dT[X](A)^K,
\end{displaymath}
and call it the {\em multiplicative transfer}.
The following lemma summarizes the properties of the additive and
multiplicative transfers that we shall need. Recall that $F_H^G \colon
\dT[X](A)^G \to \dT[X](A)^H$ is the inclusion of fixed points.
\begin{lemma}
\label{identify-transfers}
  Let $H$ and $K$ be subgroups of $G$ and let $X$ be a
  $G$-space. Let $\phi_{K,H} \colon \pi_0 \dT[X](A) \to \pi_0
  \dT[X](A)$ be the map defined by $\phi_{K,H}(x) = |(G/H)^K| \, x^{|H|/|K|}$ and let
  $q$ be the quotient map $q \colon X \to X/H$. The following diagram commutes:
  \begin{displaymath}
    \begin{CD}
      \pi_0 \dT[X](A) @>{\phi_{K,H}}>> \pi_0 \dT[X](A)
      @>{\pi_0 \dT[q](A)}>> \pi_0 \dT[X/H](A) \\
      @V{\Delta_K}VV @. @A{R^H_H}AA \\
      \pi_0 \dT[X](A)^K @>{V^K_G}>> \pi_0 \dT[X](A)^G @>{F^G_H}>>
      \pi_0 \dT[X](A)^H .
    \end{CD}
  \end{displaymath}
\end{lemma}

\begin{lemma}
\label{norm-transfer}
  Let $\mathcal F$ be a closed family of subgroups of $G$ and let $K$
  be a minimal element of $\mathcal F$. For every connected free $G$-space $X$
  the following diagram commutes:
$$
\xymatrix{
{\pi_0 \dT[X](A)}
\ar[r]^{\Delta_K}
\ar[d]_{\pi_0 \dT[q](A)}^{\cong}
&
{\pi_0 \dT[X](A)^K}
\ar[r]^{V^G_K}
\ar[d]_{\mathrm{res}}
\ar[dl]_{R^K_K}
&
{\pi_0 \dT[X](A)^G}
\ar[d]_{\mathrm{res}}
\\
{\pi_0 \dT[X/K](A)}
\ar[dr]_{\cong}^{\mathrm{can}}
&
{\pi_0 \dT[X]^A(\mathcal F)^K}
\ar[l]_{\rho}^{\cong}
\ar[r]^{V^G_K}
&
{\pi_0 \dT[X]^A(\mathcal F)^G}
\\
&
{\pi_0 \dT[X/K](A)_{hW_GK}.}
\ar[ur]
&
}
$$
  Here $q \colon X \to X/K$ is the quotient map, the map $\rho$ is
  induced by the natural isomorphism in (\ref{eq:label}) and the unlabeled map
  is given by the fundamental cofibration sequence in 
  Lemma \ref{lem:funcofseq}. 
\end{lemma}
\begin{proof}
  Apart from the lower part of the diagram everything follows directly from
  the definitions. We need to jump back to the proof of Lemma \ref{lem:fundcofseq1},
  where we in particular considered the spaces $Z(j)$. The
  commutativity of the lower part of the diagram follows from the
  commutativity of the diagram
  \begin{displaymath}
    \begin{CD}
      Map_*(Z(j),A(j))^K @>{V^G_K}>> Map_*(Z(j),A(j))^G \\
      @V{\cong}VV @V{\cong}VV \\
      Map(Z(j)^K,A(j)^K) @>{V^{W_GK}_{\trivgroup}}>> Map_*(Z(j)^K,A(j)^K)^{W_GK}
      \\
      @V{\mathrm{can}}VV @A{\mathrm{can}}AA \\
      Map_*(Z(j)^K,A(j)^K)_{hW_GK} @>{N}>> Map_*(Z(j)^K,A(j)^K)^{hW_GK}
    \end{CD}
  \end{displaymath}
  in the homotopy category of the category of $\Gamma$-spaces. Here
  $N$ is the norm map from Section \ref{sec:herecomesnorm}, and the commutativity of this diagram
  is a direct consequence of the definitions.
\end{proof}

We define the {\em extended Teichm\"uller map} $\tau^{\mathcal F}_A
\colon \Witt_{\mathcal F}(\pi_0 \dT[X](A)) \to \pi_0 \dT[X]^A(\mathcal
F)^G$ to be the composition 
\begin{displaymath}
  \Witt_{\mathcal F}(\pi_0 \dT[X](A)) \to \pi_0 \dT[X](A)^G
  \xrightarrow {\mathrm{res}}  \pi_0
  \dT[X]^A(\mathcal F)^G,
\end{displaymath}
where the first map takes $x = (x_H)_{H \le G}$ to $\sum_{[K]} V^G_K
\Delta_K(x_K)$, the sum runs over the conjugacy classes of
subgroups
in
$\mathcal F$ and the second map is the restriction map from Section
\ref{sec:herecomesnorm}.  
Taking $\mathcal F$ to be the family of all subgroups of
$G$ 
we obtain the extended Teichm\"uller map $\tau^G_A \colon \Witt_G(\pi_0
\dT[X](A)) \to \pi_0 \dT[X](A)^G$.
\begin{lemma}
\label{tauinjlem}
  Let $\mathcal F$ be a closed family of subgroups of $G$ and let $K$
  be a minimal element of $\mathcal F$. If the underlying abelian
  group of $\pi_0 A$ is torsion free, then the composite 
  $$\pi_0 \dT[X](A) \to \Witt_{\mathcal F}(\pi_0 \dT[X](A))
  \xto{\tau^{\mathcal F}_A} \pi_0 \dT[X]^A(\mathcal F)^G$$
is injective.
Here the left map is the diagonal inclusion in the factor
corresponding to $K$. 
\end{lemma}
\begin{proof}
  By Lemma \ref{identify-transfers} the composite
  \begin{eqnarray*}
    \pi_0 \dT[X](A) \xto {\Delta_K} \pi_0 \dT[X](A)^K &\xto{V^G_K}&
    \pi_0 \dT[X](A)^G \\
    &\xto{\mathrm{res}}& 
    \pi_0 \dT[X]^A(\mathcal F)^G \\
    &\xto{F^G_K}&
    \pi_0 \dT[X]^A(\mathcal F)^K = \pi_0 \dT[X](A)^K \\
    &\xto{R^K_{\trivgroup}}&
    \pi_0 \dT[X/K](A) \cong \pi_0 \dT[X](A)
  \end{eqnarray*}
  is equal to $\phi_{K,K}$, that is, it is multiplication by the
  cardinality of $W_GK$.
\end{proof}

The existence of the endofunctor $B \mapsto \Witt_G(B)$ is a corollary
of the proof of the following proposition.
 \begin{prop}\label{prop:pi0offix}
   Let $G$ be a finite group, $X$ a connected free $G$-space and $A$ a
   commutative $\ess$-algebra. 
   The extended Teichm\"uller map $\tau^G_A\colon \Witt_G(\pi_0A)\to
   \pi_0[\dT[X](A)]^G$ is an isomorphism of rings. 
 \end{prop}
 \begin{proof}
We first consider the case where $\pi_0 A$ is torsion free.
If $\mathcal F =
\{G\}$, then it is a consequence of Lemma
\ref{norm-transfer} and Lemma \ref{lem:fundcofseq1} that $\tau_A^G$ is bijective
since $\dT[X]^A(\emptyset) = *$. If $\mathcal G \subseteq \mathcal F$ are
$K$-adjacent closed families of subgroups of $G$, then by Lemma
\ref{norm-transfer} there is a commutative diagram of the form
\begin{displaymath}
  \begin{CD}
    \pi_0 \dT[X](A) @>>> \Witt_{\mathcal F}(A) @>>> \Witt_{\mathcal
      G}(A) \\
    @V{\cong}VV @V{\tau_A^{\mathcal F}}VV @V{\tau_A^{\mathcal G}}VV \\
    \pi_0 \dT[X/K](A)_{hW_GK} @>>> \pi_0 \dT[X]^A(\mathcal F)^G
    @>{\mathrm res}>>
    \dT[X](\mathcal G)^G
  \end{CD}
\end{displaymath}
Here the upper left map is the diagonal inclusion in the factor of the
product corresponding to the conjugacy class of $K$ and the upper right
map is the projection away from that factor. If $\tau_A^{\mathcal G}$ is
a bijection, then using Lemma \ref{tauinjlem} and the $5$-lemma it follows readily that also $\tau_A^{\mathcal F}$ is
a bijection. Thus by induction $\tau^{\mathcal F}_A$ is a bijection for every
$\mathcal F$, and in particular $\tau_A^G$ is a bijection. 

 When $\pi_0A$ has torsion one proceeds as follows. 

 First notice that
 $\tau^G_A$ is natural in $A$. Let $P.\wefib \pi_0A$ be
 a simplicial resolution of the discrete ring $\pi_0A$ such that $P_q$ is torsion free
 for all $q$. Then $HP. \to H\pi_0A$ is a fibration, and therefore
 $Q.=HP.\times_{H\pi_0A}A$ is a
 simplicial resolution of $A$.  Note that $\pi_0(Q_q)\cong P_q$ for every $q$.  It is
 easy to see that $\pi_0\{[q]\mapsto \Witt_G(\pi_0(Q_q))\}\cong
 \Witt_G(\pi_0A)$, and likewise $\pi_0\{[q]\mapsto
 \pi_0[\dT[X](Q_q)]^G\}\cong\pi_0[\dT[X](A)]^G$ (for the last piece one
 must show that $A\mapsto [\dT[X](A)]^G$ may be ``calculated
 degreewise'', a fact which
 is readily proved by induction using the
 fundamental cofibration sequence).
 Thus, for each $q$ $\tau^G_{Q_q}$ is an isomorphism, and so $\tau^G_A=\pi_0\tau^G_{Q.}$
 is an isomorphism.

Now it
follows from Lemma \ref{identify-transfers} that the composition 
\begin{multline*}
  \Witt_G(\pi_0 A) \cong \Witt_G(\pi_0 \dT[X](A)) \xto {\tau_A^G} \pi_0
  \dT[X](A)^G 
  \xto {F^G_H} \pi_0 \dT[X](A)^H \xto {R^H_{\trivgroup}} \pi_0
  \dT[X/H](A) \cong \pi_0(A)
\end{multline*}
is equal to $\phi_H$. Taking $A$ to be the Eilenberg MacLane spectrum
on a commutative ring $B$ we see that the ring structure on $\Witt_G(\pi_0 A)$
obtained from the bijection to $\pi_0 \dT[X](A)^G$ satisfies the two
criteria mentioned in 
\ref{subseq:BWring}
determining the ring structure on the ring $\Witt_G(B)$ of
Witt vectors on $B$ 
uniquely. Thus we may conclude the existence of the ring
$\Witt_G(B)$ of Witt
vectors and moreover that   $\pi_0 \dT[X](A)^G$ is isomorphic to
$\Witt_G(\pi_0 A)$.
 \end{proof}

Dress and Siebeneicher also define Burnside-Witt rings over profinite
groups. In order to recall their construction we note that, if $N$ is a
normal subgroup of a finite group $G$ and if $f \colon G \to G' =
G/N$ is the quotient homomorphism, then there
is a natural (surjective) ring homomorphism $R^G_N \colon \Witt_G(B) \to
\Witt_{G'}(B)$ taking $x = (x_H)_{H \le G}$ to $y = (y_{H'})_{H' \le
  G'}$ where $y_{H'} = x_{f^{-1}H}$. 
\begin{lemma}
\label{comparerestr}
  Let $A$ be a commutative $\ess$-algebra, let $N$ be a normal
  subgroup of a finite group $G$, let $X$ be a $G$-space and let $q
  \colon X \to X/N$ be the quotient map. For every subgroup $K$ of $G$ containing $N$
  the diagram
  \begin{displaymath}
    \begin{CD}
      \pi_0 \dT[X]A @>{\pi_0 \dT[q]A}>> \pi_0 \dT[X/N] A \\
      @V{\Delta_K}VV @V{\Delta_{K/N}}VV \\
      \pi_0 \dT[X]A^K @>{R^K_N}>> \pi_0 \dT[X/N] A^{K/N} \\
      @V{V^G_K}VV @V{V^{G/N}_{K/N}}VV \\
      \pi_0 \dT[X]A^G @>{R^G_N}>> \pi_0 \dT[X/N] A^{G/N}       
    \end{CD}
  \end{displaymath}
  commutes. 
\end{lemma}

Given a profinite group $G$, we let $\Witt_G(B)$ be the projective
limit of the rings $\Witt_{G/U}(B)$ with respect to the homomorphisms
$R^{G/U}_{V/U}$ for $U \le V$ open (i.e. finite index) normal subgroups of $G$.
Let $X$ be a $G$-space with the property that $X/U$
is a free connected $G/U$-space for every open normal subgroup $U$ in $G$. (The
guiding example we have in mind is when $G = (\bZ^{\wedge})^{\times n}$ is the profinite
completion of $\bZ^n$ or $G = (\bZ_p)^{\times n}$ is the
$p$-completion of $\bZ^n$ and $X = G \times_{\bZ^n} \R^n$.) Then it
follows from Lemma \ref{comparerestr} 
that the diagram
\begin{displaymath}
  \begin{CD}
    \Witt_{G/U} (\pi_0 \dT[X/U]A) @>{\tau^{G/U}_A}>>  \pi_0 \dT[X/U]
    A^{G/U} \\
    @V{R^{G/U}_{V/U} \circ \Witt_{G/U}(\pi_0 \dT[q]A)}VV @V{R^{G/U}_{V/U}}VV \\
    \Witt_{G/V} (\pi_0 \dT[X/V]A) @>{\tau^{G/V}_A}>>  \pi_0 \dT[X/V] A^{G/V} 
  \end{CD}
\end{displaymath}
commutes for $U \le V$ open normal subgroups of $G$. Since the Mittag
Leffler condition is satisfied we conclude that there are isomorphisms
\begin{displaymath}
  \Witt_G(\pi_0A) \cong \lim_U \Witt_{G/U}(\pi_0A) \cong \lim_U \pi_0 \dT[X/U]
    A^{G/U} \cong \pi_0 \holim[U] \dT[X/U]
    A^{G/U}
\end{displaymath}
where the limits 
are taken over restriction maps.
We summarize the above discussion with a very special case. 

Given a prime $p$, consider the set $\mathcal O_p$ consisting of
  subgroups $U\subseteq\Z^n$ with index a power of $p$.  Choose a free contractible $\bZ^n$-space $E$
  (\eg $\sin\R^n$), and
consider the diagram of spaces $E/U$ where $U$ varies over $\mathcal
O_p$.  Notice that this diagram is equivalent to a diagram of isogenies
of the $n$-torus. 
\begin{cor}
  Let $A$ be a connective commutative $\ess$-algebra.
Then there is a natural ring isomorphism ``preserving $F$ and $V$''
between Dress and Siebeneicher's Burnside Witt ring $\Witt_{\Z_p^{\times
    n}}(\pi_0A)$ and $\pi_0 \holim[U\in\mathcal O_p] (\dT[E/U] A)^{\Z^n/U}$.
\end{cor}
Note that we have a cofinal subsystem given by the powers of $p$, so
that both groups in the corollary are
isomorphic to $\pi_0 \holim[k] (\dT[E/p^k\Z^n] A)^{\Z^n/p^k\Z^n}$.

\setcounter{subsection}{0}\setcounter{subsubsection}{0}
\section{Covering Homology}
\label{sec:covhom}

In this section we define covering homology through appropriate homotopy limits of the restriction
and Frobenius maps.  The name derives from the fact that in our main
examples the limit is indexed by systems of self-coverings.

\subsection{Spaces with finite actions}
\label{defineEE}
We shall use a category $\EE$ to index the homotopy limit defining covering homology.
The objects of $\EE$ consist of triples $(G,H,X)$ where $G$ is a
simplicial group, $H$ is a discrete subgroup of $G$ and $X$ is a
non-empty
$G$-space with the property that the image of $H$ in $\Aut(X)$ is finite.
A morphism in
$\EE$ from $(G,H,X)$ to $(G',H',X')$ consists of a pair
$(\varphi,f)$ where $\varphi \colon G \to G'$ is a group homomorphism
 such that $H' \subseteq \varphi (H)$
and $f \colon X \to \varphi^* X'$ is a $G$-map. The composition in $\EE$
is given by composing in each of the two entries.

Covering homomogy is built from functors into $\EE$. Before we give examples of functors into $\EE$ we
recall the definition of the twisted arrow category.
\begin{Def}
  If $\cC$ is a category, then the {\em twisted arrow category}
  $\Ar(\cC)$ of
  $\cC$ is the category whose objects are arrows in $\cC$. A
  morphism from $d\colon x\to y$ to $b\colon z\to w$ in $\Ar(\cC)$ is a commutative
  diagram 
  \begin{displaymath}
  \begin{CD}
    x @>c>> z 
    \\ 
    @V{d}VV @V{b}VV
    \\ 
    y @ < {a} < < w
  \end{CD}    
  \end{displaymath}
in $\cC$, \ie every equation $abc=d$ represents an arrow
$(a,c)$ from $d$ to
$b$, and composition is horizontal composition of squares:
$(a_1,c_1)(a_0,c_0)=(a_0a_1,c_1c_0)$.  
\end{Def} 
The shorthand notation $a^*=(a,id)$ and
$c_*=(id,c)$ is usual, giving formulae like $a^*c_*=c_*a^*$,
$(ab)^*=b^*a^*$ and $b_*c_*=(bc)_*$.

{
  Given a simplicial group $G$ we let $I(G)$ be the monoid of isogenies
  of $G$, that is, 
  group-endomorphisms of $G$ with finite discrete
  kernel and cokernel. 
\begin{Def}
\label{artoe}
  Let $G$ be a simplicial group and consider $I(G)$ as a category with
  one object. Let $\Ar_G$ be the subcategory of $\Ar(I(G))$ containing
  all objects and with morphism set $\Ar_G(\delta,\beta)$ equal to
  the set of pairs $(\gamma,\alpha)$ of isogenies of $G$ with the
  property that $\Ker \beta \subseteq \gamma(\Ker \delta )$. 
  The
  functor $S_G \colon \Ar_G \to \EE$ takes an object $\alpha \colon G
  \to G$ of $\Ar_G$ to the triple
  $S_G(\alpha) = (G,\Ker(\alpha),X)$, where $X=G$ is the $G$-space where
  $G$ acts on itself by translation.
  A morphism
  \begin{displaymath}
  \begin{CD}
    G @>\gamma>> G 
    \\ 
    @V{\delta}VV @V{\beta}VV
    \\ 
    G @ < {\alpha} < < G
  \end{CD}    
  \end{displaymath}
  in $\Ar_G$ is taken to the morphism
  $$S_G(\alpha,\gamma) = (\gamma,f) \colon (G,\Ker(\delta),X) \to (G,\Ker(\beta),X)$$
  in $\EE$, where $f \colon X \to \gamma^* X$ is equal to $\gamma$
  considered as a map of $G$-spaces. 
\end{Def}
}
{

Given a set $X$ we let $EX$ denote the contractible
simplicial set defined by the formula $EX_k =
Map([k],X)$, where $[k] \in \Delta$ is considered as a set with $k+1$
elements. Note that there is an inclusion $X \cong EX_0 \subseteq EX$.
If $X = K$ is a discrete group, then $EK = EX$ is a simplicial group
containing $K$. In particular, $K$ acts freely on $EK$. Note that if
$K$ is abelian, then $EK$ is abelian. 

  Given an inclusion $H \subseteq K$ of discrete abelian groups,
  let $\CC(H,K)$ be the monoid of group automorphisms $a
  \colon K \to K$ with the property that $a(H)$ is a subgroup of finite index in $H$.
  There is a monoid homomorphism $\varphi \colon \CC(H,K) \to I(EK/H)$,
taking $a \in
  \CC(H,K)$ to the surjection 
  \begin{displaymath}
    \varphi(a) \colon EK/H \to EK/H
  \end{displaymath}
  induced by $a$.

  In the situation of Definition \ref{artoe}, let $G = EK/H$. Since the functor $\Ar(\varphi)
  \colon \Ar(\CC) \to \Ar(I(G))$ factors through the subcategory
  $\Ar_G$ we can make the following definition.

\begin{Def}
\label{S(C)}
   Let $H \subseteq K$ be an inclusion of discrete abelian groups.
   Given a submonoid $\CC$ of $\CC(H,K)$ we let
   $S(\CC) \colon \Ar(\CC) \to \EE$ be the functor taking $c \in \CC$ to the triple
  $S(\CC)(c) = (EK/H,\Ker(\varphi(c)),EK/H)$.
\end{Def}
}
\subsection{Covering homology}
To every commutative $\ess$-algebra $A$ there is an associated functor
$\dT[]A$ from the category $\EE$ of \ref{defineEE} to the category of spectra, taking an object
$(G,H,X)$ to $\left[ \dT[X]A \right]^H$. 
For a morphism $(\varphi,f) \colon (G,H,X) \to (G',H',X')$ in $\EE$ we write
$K = \Ker(\varphi) \cap H$, and by abuse of notation we write $\dT[f]$
for the map
\begin{displaymath}
  \left[ \dT[X/K]A  \right]^{H/K} \to \left[ \dT[\varphi^* X']
    A\right]^{H/K} = \left[\dT[X'] A \right]^{\varphi(H)} 
\end{displaymath}
induced by $f$.
We define
the map $\dT[]A(f,\varphi)$ to be the composition
\begin{displaymath}
  \left[ \dT[X]A \right]^H \xto{R_K^H} \left[ \dT[X/K]A  \right]^{H/K} 
  \xto {\dT[f]} \left[ \dT[X'] A\right]^{\varphi(H)}
  \xto {F_{H'}^{\varphi(H)}}
  \left[\dT[X'] A \right]^{H'}. 
\end{displaymath}
In order to check that $\dT[]A$ is a functor, let 
\begin{displaymath}
 (\varphi,f) \colon (G,H,X) \to (G',H',X') 
\end{displaymath}
and 
$$(\psi,g) \colon (G',H',X') \to (G'',H'',X'')$$
be morphisms in $\EE$. Let $K = \Ker(\varphi) \cap H$, $K' =
\Ker(\psi) \cap H'$ and $K'' = \Ker(\psi\varphi) \cap H$.
By the naturality and transitivity
properties of the maps
$R$ and $F$ stated in Lemma \ref{lem:RF=FR} the diagram
\begin{displaymath}
\xymatrix{
{\left[\dT[X]A\right]^H}
\ar[dr]^{R}
\ar[d]_{R} 
\\
\ar[r]^R
\ar[d]_{\dT[f]}
{\left[\dT[X/K]A\right]^{H/K}}
 &
\ar[d]_{\dT[f]}
{\left[\dT[X/K'']A\right]^{H/K''}}
\ar[dr]^{\dT[gf]}
\\
\ar[r]^R
\ar[d]_F
{\left[\dT[X']A\right]^{H/K}}
 &
\ar[d]_F
\ar[r]^{\dT[g/K']}
{\left[\dT[X'/K']A\right]^{H/K''}}
&
\ar[d]_F
\ar[dr]^F
{\left[\dT[X'']A\right]^{H/K''}}
\\
\ar[r]^R
{\left[\dT[X']A\right]^{H'}}
 &
\ar[r]^{\dT[g/K']}
{\left[\dT[X'/K']A\right]^{H'/K'}}
&
{\left[\dT[X'']A\right]^{H'/K'}}
\ar[r]^F
&
{\left[\dT[X'']A\right]^{H''}}
}
\end{displaymath}
commutes. 
\begin{Def}
  Let $A$ be a commutative $\ess$-algebra and let $S \colon \A \to
  \EE$ be a functor from an arbitrary category $\A$.
  The {\em covering homology} of $A$ with respect to the functor $S$
  is 
  \begin{displaymath}
    TC_S(A) := \holim[\A] \dT[]A \circ S 
  \end{displaymath}
\end{Def}
{
  Let $G$ be a simplicial group and let $I(G)$ be the 
  monoid of isogenies of $G$, that is, group-endomorphisms $\alpha$ of $G$ with
  finite and discrete kernel and cokernel. { 
  In Definition \ref{artoe} we constructed a functor $S_G \colon
  \Ar_G \to \EE$. Thus there is a covering homology
  $A \mapsto TC_{S_G}(A)$ associated to every simplicial
  group. The fact that surjective isogenies are simplicial
  versions of finite covering maps is the reason for our choice of the name 
  ``covering homology''.
 Note that
  given $\alpha$ and $\beta$ in $I(G)$, the map $\dT[\alpha]$ is the composite
\begin{displaymath}
   [\dT[G/\Ker{\alpha}] A]^{\Ker{(\beta
      \alpha)}/\Ker{\alpha}}  \to \left[ \dT[\alpha^* G] 
    A\right]^{\Ker{(\beta
      \alpha)}/\Ker{\alpha}} = \left[\dT[G] A \right]^{\Ker(\beta)}.
\end{displaymath}
If $\alpha$ is surjective, the map $G/\Ker(\alpha) \to G$ induced
by $\alpha$ is an isomorphism, and thus $\dT[\alpha]$ is an
isomorphism in this case. 

We shall occasionally use the notations $R_\gamma = \dT A(\gamma_*)$
and $F^\alpha = \dT (\alpha^*)$ for $\dT A$ applied to morphisms of
the form
\begin{displaymath}
\gamma_* =
  \begin{CD}
    G @>{\gamma}>> G \\
    @VV{\delta}V @VV{\beta}V \\
    G @<{=} < < G
  \end{CD}
\qquad \text{and} \qquad 
\alpha^* = 
  \begin{CD}
    G @>{=}>> G \\
    @VV{\delta}V @VV{\beta}V \\
    G @<{\alpha} < < G
  \end{CD}
\end{displaymath}
in $\Ar_G$.

\begin{ex}
\label{CCex}
\label{tcex}
  The situation where $G$ is the simplicial
  $n$-torus $\TZ=\sin(\R^n/\Z^n)$ is
  particularly interesting. There is an isomorphism
  \begin{displaymath}
    Hom(\Z^n,\Z^n) \to Hom(\R^n/\Z^n,\R^n/\Z^n)
  \end{displaymath}
  taking $a$ to the map induced by $\R \otimes_\Z a \colon \R^n \to
  \R^n$, and the induced homomorphism
  \begin{displaymath}
    Hom(\Z^n,\Z^n) \subseteq Hom(\TZ,\TZ), \quad a \mapsto
    \sin(\R\otimes_\Z a),
  \end{displaymath}
  is injective.
  This allows us to consider the monoid 
  $\M_n$ of injective linear
  endomorphisms of $\Z^n$ as a submonoid of $I(G)$. Thus, given a submonoid $\CC$ of $\M_n$, we
  can consider the covering homology $A \mapsto TC_{S(\CC)}(A)$. 

In
  the special case where $n=1$ and where $\CC=\M_{1} = (\mathbb
  Z \setminus \{0\},\cdot)$, the covering
  homology $TC_{S(\M_{1})}(A)$ gives us what might be called {\em
  topological dihedral homology}. If $\CC$ is the submonoid 
  $(\mathbb N_{>0},\cdot)$, giving just orientation-preserving
  coverings of the circle, the covering homology $TC_{S(\CC)}(A)$ is
  weakly equivalent to B\"okstedt,
  Hsiang and Madsen's topological cyclic homology. 
\end{ex}
}}
Let $G$ be a simplicial group, let $\CC$ be a submonoid of $I(G)$ and
let $\Ar_{\CC} := \Ar(\CC) \cap \Ar_G$. 
If $\varphi \colon G \to G$ is a group automorphism with the property that 
$\varphi x \varphi^{-1} \in \CC$ for every $x \in \CC$, then
we define the functor
\begin{displaymath}
  c_{\varphi} \colon \Ar_\CC \to \Ar_\CC, \quad x \mapsto \varphi
  x \varphi^{-1},
\end{displaymath}
and the commutative diagram
  \begin{displaymath}
    \begin{CD}
      G @>{\varphi}>> G \\
      @V{x}VV @VV{\varphi x \varphi^{-1}}V \\
      G @< { \varphi^ {-1} } < <  G,
    \end{CD}
  \end{displaymath}
specifies a natural transformation
$\eta_{\varphi}$ from the inclusion $j \colon \Ar_\CC \subseteq
  \Ar_G$ 
  to the functor $j \circ c_{\varphi}$.
Since $(\varphi,\varphi^{-1}) = \varphi^* \circ \varphi^{-1}_*$ we
have 
$$\dT A(\varphi,\varphi^{-1}) = \dT A(\varphi^*) \circ \dT A
(\varphi^{-1}_*) = R_\varphi \circ F_{\varphi^{-1}} \colon \dT[G] A
\to \dT[G] A.$$ 
Let $S_\CC$ be the composite
\begin{displaymath}
  \Ar_{\CC} \xto j \Ar_G \xto{S_G} \EE.
\end{displaymath}
\begin{Def}
\label{tcaction}
  Let $G$ be a simplicial group, let $\CC$ be a submonoid of $I(G)$
  and let $I$ be a group contained in $I(G)$ with the property that
$\varphi x \varphi^{-1} \in \CC$ for every $x \in \CC$ and every $\varphi
\in I$. We define an action of $I$ on $TC_{S_\CC}(A)$ by letting
the action of $\varphi \in I$ by given by the endomorphism 
\begin{displaymath}
\holim \Lambda A \circ S_{\CC} \xto {\holim \Lambda A \circ S_G \circ \eta_\varphi}
\holim \Lambda A \circ S_{\CC} \circ c_{\varphi} \to
\holim \Lambda A \circ S_{\CC}. 
\end{displaymath}
\end{Def}
Note that The group $I$
  acts through maps.

%
{
  \begin{ex}
    Let $G = \TZ$, let $p$ be a prime and let $\CC$ denote the submonoid of $I(G)$
    consisting of isogenies of the form
    \begin{displaymath}
      p^r \colon \TZ \to \TZ, \qquad (x_1,\dots,x_n) \mapsto p^r(x_1,\dots,x_n)=(p^rx_1,\dots,p^rx_n).
    \end{displaymath}
    for $r \ge 0$ corresponding to the submonoid of $\M_n$ consisting of endomorphisms
    of $\Z^n$ of the form $p^r \colon \Z^n \to \Z^n$. The group of
    automorphisms of $\Z^n$ is contained in $\M_n$. Since the
    endomorphisms in $\CC$ correspond to multiplication by a number,
    they are fixed under conjugation by elements in the group $GL_n
    \Z$ of group automorphisms of $\Z^n$. Thus Definition
    \ref{tcaction} specifies an action of $I$ on $TC_{S_\CC}A$. In this
    particular situation the category
    \begin{displaymath}
      \xymatrix{\dots \ar@<1ex>[r]^{p^*}  \ar@<-1ex>[r]_{p_*} &   (p^2) \ar@<1ex>[r]^{p^*}
        \ar@<-1ex>[r]_{p_*} &   (p) \ar@<1ex>[r]^{p^*}  \ar@<-1ex>[r]_{p_*} & (1) 
        },
    \end{displaymath}
   with $(p^n)$ equal to the endomorphism of $G$ given by multiplication by
   $p^n$, is cofinal in the category $\Ar_{\CC}$. Thus $TC_{S_{\CC}}$ is
   homotopy equivalent to the homotopy limit of the diagram
    \begin{displaymath}
      \xymatrix{ \dots \qquad \ar@<1ex>[r]^{F_p}  \ar@<-1ex>[r]_{R_p} &   {(\dT[\TZ] A)^{C_{p^2}^{\times n}}} \ar@<1ex>[r]^{F_p}
        \ar@<-1ex>[r]_{R_p} &   {(\dT[\TZ] A)^{C_p^{\times n}}}  \ar@<1ex>[r]^{F_p}
        \ar@<-1ex>[r]_{R_p} & {\dT[\TZ] A} 
        }.
    \end{displaymath}

If $n=1$ we may consider the action of $\{\pm1\}=GL_1(\Z)$, and the
homotopy fixed point spectrum picks up the ``part relevant to $p$'' of the
{
  topological dihedral homology} in Example \ref{tcex}. 

Working with the $p$-complete torus instead, we get operations by all of
$GL_n(\Z_p)$.  Note that if $n=1$ the map from $TC(A)\p\simeq TC_{S_\CC}(A)\p$
to $THH(A)\p\simeq \dT[(\T^1)\p](A)\p$ then sends the operation of a
$p$-adic unit on $TC(A)\p$ to the corresponding Adams operation on
$THH(A)\p$, as discussed in \ref{adams}.
  \end{ex}
}
\begin{ex}
  \label{ringincl}
  Let $R \subseteq B$ be an inclusion of (discrete) commutative
  rings and
  let $M$ be a flat $R$-module. In the context of Definition
  \ref{S(C)}, let $H =
  M = R \otimes_R M$ let $K = B
  \otimes_R M$ and let $H \subseteq K$ be the inclusion of discrete
  abelian groups induced by the $R$-module homomorphism $R \to B$. 
  Applying  Definition \ref{S(C)} to a submonoid $\CC$ of $\CC(M,B
  \otimes_R M)$
  we obtain a functor $S(\CC) \colon 
  \Ar(\CC) \to \EE$, and we may form the covering homology $A \mapsto TC_{S(\CC)}(A)$. 
  If $I$ is a group contained in $\CC(M,B
  \otimes_R M)$ 
  with the
  property, that 
  $\varphi x \varphi^{-1} \in \CC$ for every $x \in \CC$ and every $\varphi
  \in I$, then Definition \ref{tcaction} specifies an action of $I$ on $TC_{S(\CC)}(A)$.   

  Let us emphasize that if $R \subseteq B$ is the inclusion
  $\Z \subseteq \Q$ and if $M = \Z^n$, then $G = E(\Q \otimes_\Z
  M)/(\Z \otimes_\Z M)$ is a model for the classifying space
  $B\Z^n$. In fact, the homomorphisms
  \begin{displaymath}
    \R^n/\Z^n \gets (\R^n \times |E(\Q \otimes_\Z \Z^n)|)/\Z^n \to |E(\Q \otimes_\Z \Z^n)|/\Z^n \cong |G| 
  \end{displaymath}
  of topological abelian groups are homotopy equivalences. In Example
  \ref{tcex} we have seen that under these
  equivalences $\M_n = \CC(\Z^n,\Q^n)$ corresponds to the monoid of isogenies of the
  $n$-torus $\R^n/\Z^n$.
  The spectrum $TC_{S(\CC(\Z^n,\Q^n))}(A)$ is related to iterated topological
  cyclic homology. In fact, in Example \ref{tcex} we have seen that when
  $n=1$ and 
  $\CC$ is 
  the submonoid $(\mathbb N_{>0},\cdot)$ of $\CC(\Z,\Q) = (\Z
  \setminus \{0\},\cdot)$, then $TC_{S(\CC)}(A)$ is weakly
  equivalent to B\"okstedt,
  Hsiang and Madsen's topological cyclic homology. Note that
  Definition \ref{tcaction} gives an action of
  $I = \{-1,+1\}$ on $
  TC_{S(\CC)}(A)$ whose homotopy fixed point spectrum is the 
{topological dihedral homology} of Example\ref{tcex} .

  Consider the situation where $B$ is the quotient field of an integral
  domain $R$ and $M=\cO$
  is a possibly non-commutative $R$-algebra. In this situation  
  we can choose the monoid $\CC$ to be the intersection of $\CC(\cO, B
  \otimes_R \cO)$ and 
  image of the homomorphism $\psi \colon \cO \to End_\Z(B \otimes_R
  \cO)$ from $\cO$ to the
  monoid $End_\Z(B \otimes_R \cO)$ of group-endomorphisms of $B
  \otimes_R \cO$ with $\psi(x)(b \otimes y) = b \otimes yx$.
  If $f
  \colon \cO \to \cO$ is an $R$-algebra automophism,
  then the diagram
  \begin{displaymath}
    \begin{CD}
      B \otimes_R \cO @>{B \otimes_R f}>{\cong}> B \otimes_R \cO \\
      @V{\psi(x)}VV @V{\psi(f(x))}VV \\
      B \otimes_R \cO @>{B \otimes_R f}>{\cong}> B \otimes_R \cO 
    \end{CD}
  \end{displaymath}
  commutes. This implies that
  if we let $I$ be the group of $R$-algebra automorphisms of $B
  \otimes _R \cO$, of the form $\varphi = B \otimes_R f$, then
  $\varphi \psi(x) \varphi^{-1} \in \cO$ for every $\varphi \in I$ and
  every $\psi(x) \in \CC$. Thus the group of $R$-algebra automorphisms
  of $\cO$ acts on $TC_{S(\CC)}(A)$. 

  Explicit examples are listed in the figure below, where $G$ is
  a finite group, $K \subseteq L$ is a finite Galois extension of
  (local) number fields and $\cO(K) \subseteq \cO(L)$ is the induced inclusion
  of rings of integers.
  \begin{center}
  \begin{tabular}{l|l|l|l|l}
    $R$ & $B$ & $\cO$ & $\CC$ & $\Aut_R(\cO)$ \\ \hline
    $\Z$ & $\Q$ & $\Z^n$ & $(\Z \setminus \{0\})^n$ & $\Sigma_n$ \\
    $\Z\p$ & $\Q_p$ & $(\Z\p)^n$ &  $(\Z\p \setminus \{0\})^n$ &
    $\Sigma_n$ \\
    $\Z$ & $\Q$ & $\Z[G]$ & $\Z[G] \cap \Q[G]^*$ & $\Aut(G)$ \\
    $\cO(K)$ & $K$ & $\cO(L)$ & $\cO(L) \cap L^*$ & $Gal(L/K)$ \\ 
  \end{tabular}    
  \end{center}
\end{ex}

\bibliographystyle{plain}

\end{document}